\documentclass[12pt,reqno]{amsart}
\usepackage{amsmath,amssymb,amsfonts,amsthm,amscd,amstext,amsxtra,amsopn,array,url,verbatim,mathrsfs}
\usepackage{lmodern}
\usepackage{enumerate}
\usepackage{bm}
\usepackage{geometry}

\usepackage[mathscr]{euscript}
\usepackage{pdflscape}

\usepackage{graphicx}
\usepackage{verbatim}
\usepackage{times}
\usepackage[colorlinks=true,linkcolor=black,citecolor=black]{hyperref}

\usepackage{fullpage}

\makeatletter
\@namedef{subjclassname@2010}{%
  \textup{2010} Mathematics Subject Classification}
\makeatother

\newtheorem*{rep@theorem}{\rep@title}
\newcommand{\newreptheorem}[2]{%
\newenvironment{rep#1}[1]{%
 \def\rep@title{#2 \ref{##1}}%
 \begin{rep@theorem}}%
 {\end{rep@theorem}}}
\makeatother

\newcolumntype{M}{>{\vphantom{\large HAP}$}l<{$}}

\theoremstyle{plain}
\newtheorem{thm}{Theorem}[section]
\newtheorem{prop}[thm]{Proposition}

\newtheorem{lem}[thm]{Lemma}
\newtheorem{cor}[thm]{Corollary}

\theoremstyle{definition}
\newtheorem{rem}[thm]{Remark}
\newtheorem{exam}[thm]{Example}

\theoremstyle{plain}
\newtheorem{defn}[thm]{Definition}

\numberwithin{equation}{section}
\numberwithin{table}{section} 
\numberwithin{figure}{section}

\newreptheorem{theorem}{Theorem}
\newreptheorem{lemma}{Lemma}

\newcommand{\NN}{\mathbb{N}}
\newcommand{\ZZ}{\mathbb{Z}}
\newcommand{\QQ}{\mathbb{Q}}

\newcommand{\CC}{\mathbb{C}}
\newcommand{\FF}{\mathbb{F}}
\newcommand{\kk}{{K}}

\newcommand{\p}{\mathbf{p}}
\newcommand{\bb}{\mathbf{b}}
\newcommand{\q}{\mathbf{q}}
\newcommand{\rr}{\mathbf{r}}
\newcommand{\s}{\mathbf{s}}
\newcommand{\ee}{\mathbf{e}}
\newcommand{\vv}{\mathbf{v}}
\newcommand{\uu}{\mathbf{u}}
\newcommand{\w}{\mathbf{w}}
\newcommand{\ww}{\mathbf{w}}
\newcommand{\zz}{\mathbf{z}}

\newcommand{\pp}{\mathfrak{p}}
\newcommand{\OO}{\mathcal{O}_\nu}

\newcommand{\univ}{ \text{univ}}
\newcommand{\RU}{ \mathcal{R}}

\DeclareMathOperator{\Frac}{Frac}

\newcommand{\remove}[1]{}
\newcommand{\st}{: \, }
\newcommand{\Bas}{\mathscr{B}}

\newcommand{\Omegabar}{\overline{\Omega}}
\newcommand{\eps}{\varepsilon}

\newcommand{\Psihat}{\hat{\Psi}}
\newcommand{\epshat}{\hat{\varepsilon}}

\begin{document}

\title{On Symmetries of Elliptic Nets and Valuations of Net Polynomials}

\author{Amir Akbary, Jeff Bleaney,  and Soroosh Yazdani}

\thanks{Research of the authors is partially supported by NSERC}

\date{\today}



\address{Department of Mathematics and Computer Science \\
        University of Lethbridge \\
        Lethbridge, AB T1K 3M4 \\
        Canada}
\email{amir.akbary@uleth.ca}

%\address{Department of Mathematics and Computer Science \\
%        University of Lethbridge \\
%        Lethbridge, AB T1K 3M4 \\
%        Canada}

\email{jeff.bleaney@uleth.ca}

\email{syazdani@gmail.com}

\begin{abstract} 
Under certain conditions, we prove that the set of zeros of an elliptic net forms an Abelian group.
    We present two applications of this fact. Firstly we give a
    generalization of a theorem of Ayad on valuations of division polynomials in
    the context of net polynomials. Secondly we generalize a theorem of Ward on
    symmetry of elliptic divisibility sequences to the case of elliptic nets.
\end{abstract}

% \thanks{Both authors are supported in part by NSERC Discovery grants.}

\maketitle

\vspace{-.75cm}
\tableofcontents
%\titlecontents{subsection}[3.8em]{\contentslabel{2.3em}}{\hspace*{-2.3em}}{\titlerule*[1pc]{.}\contentspage}
\vspace{-.75cm}

\section{Introduction}

Let $E$ be an elliptic curve defined over a field $K$ with the Weierstrass model $f(x, y)=0$, where 
\begin{equation}
    \label{WE}
    f(x,y):=y^2+a_1 xy+a_3 y- x^3-a_2 x^2-a_4 x-a_6;
    ~~~a_i\in K.
\end{equation}
It is known that there are polynomials $\phi_n,~\psi_n$, and $\omega_n\in
K[x, y]/\langle f(x,y) \rangle$ such that
for any $P \in E(K)$, the group of $K$-rational points of $E$, we have
\begin{equation}
    \label{eqn DivPoly}
    nP=\left(\frac{\phi_n(P)}{\psi_n^2 (P)}, \frac{\omega_n(P)}{\psi_n^3(P)} \right).
\end{equation}
Moreover, $\psi_n$ satisfies the recursion 
\begin{equation}
    \label{dpr}
    \psi_{m+n}\psi_{m-n} =  \psi_{m+1}\psi_{m-1}\psi_{n}^{2} - \psi_{n+1}\psi_{n-1}\psi_{m}^{2},
\end{equation}
with initial conditions
\begin{align*}
    &\psi_1=1,~\psi_2=2y+a_1 x+a_3,~ \psi_3=3x^4+b_2 x^3+3b_4 x^2+3 b_6 x+b_8, \\
    &\psi_4 =\psi_2 \cdot \left(2x^6+b_2 x^5+5b_4 x^4+10 b_6 x^3+(b_2b_8-b_4b_6)x+(b_4b_8-b_6^2)\right).
\end{align*}
Here
\begin{align*}
&b_2=a_1^2+4a_2,~b_4=2a_4+a_1 a_3,~b_6=a_3^2+4a_6,\\
&b_8=a_1^2 a_6+4a_2 a_6-a_1 a_3 a_4+a_2 a_3^2-a_4^2.
\end{align*}
The polynomial $\psi_n$ is called the \emph{$n$-th division polynomial} associated to $E$. 
(See \cite[Chapter 2]{Lang} for the basic properties of division polynomials.) 

Now let $K$ be a field with a discrete valuation $\nu$, let
$\OO = \{x \in K \st \nu(x) \geq 0\}$ and $\pp = \{x \in K \st \nu(x)>0 \}.$
In \cite[Theorem A]{Ayad}, Ayad proved the following theorem on the
valuation of $\psi_n(P)$. 
\begin{thm}[{\bf Ayad}] 
    \label{ayad}
    %Let $K$ be a field with a discrete valuation $\nu$.
    Let $E/K$ be an elliptic curve defined by the polynomial \eqref{WE} with
    $a_i \in \OO$ for $i=1, 2, 3, 4, 6$. Let $P\in E(K)$ be a point in $E(K)$ such
    that $P \not\equiv \infty \pmod \pp$. Then the following
    assertions are equivalent: 
    \begin{enumerate}[(a)]
        \item  $\nu(\psi_2(P))$ and $\nu(\psi_3(P))>0$.
        \item For all integers $n\geq 2$, we have $\nu(\psi_n(P))>0$.
        \item There exists an integer $n_0\geq 2$ such that
            $\nu(\psi_{n_0}(P))$ and $\nu(\psi_{n_0+1}(P))>0.$
        \item There exists an integer $m_0\geq 2$ such that
            $\nu(\psi_{m_0} (P))$ and 
            $\nu(\phi_{m_0} (P))>0.$  
        \item Reduction of $P$ modulo $\pp$ is singular.
    \end{enumerate} 
\end{thm}

An important ingredient of the proof of the above theorem is the recursion
\eqref{dpr}. Generally, any solution over an arbitrary integral domain $R$ of the recursion
%Generally, for any integral domain $R$,
%any solution in $R$ of the recursion 
\begin{equation}
    \label{eds1} 
    W_{m+n}W_{m-n} W_{1}^{2}=  W_{m+1}W_{m-1}W_{n}^{2} - W_{n+1}W_{n-1}W_{m}^{2}, 
\end{equation} 
where $m, n \in \mathbb{Z}$, is called an \emph{elliptic sequence}.
Hence the sequence $(\psi_n(P))$ is an example of an elliptic sequence. The
theory of elliptic sequences was developed  by Morgan Ward in 1948.  An
\emph{elliptic divisibility sequence} (EDS) is an integer elliptic sequence
$(W_{n})$, which is also a divisibility sequence (i.e. $W_m\mid W_n$ if $m\mid n$). 

Theorem \ref{ayad} has an immediate application to elliptic denominator
sequences, which we will define now. Let $E/\mathbb{Q}$ be an
elliptic curve defined by \eqref{WE}, with $a_{i} \in \ZZ$ for $i = 1,2,3,4,6$,
and let $P\in E(\mathbb{Q})$ be a non-torsion point. It is known that
$$P=\left(\frac{A_P}{D_P^2}, \frac{B_P}{D_P^3} \right)$$ 
with $\gcd (A_P, D_P)= \gcd(B_P, D_P)=1$ and $D_P\geq 1$ (see \cite[Proposition 7.3.1]{Cohen}). 
Let $(D_{nP})$ be the sequence of denominators of the multiples of $P$. 
More precisely $D_{nP}$ is given by the identity 
\begin{equation}
    \label{eqn EllDen}
    nP=\left(\frac{A_{nP}}{D_{nP}^2}, \frac{B_{nP}}{D_{nP}^3} \right) 
\end{equation}
with $\gcd(A_{nP}, D_{nP})= \gcd(B_{nP}, D_{nP})=1$ and $D_{nP}\geq 1$.  One
can show that $(D_{nP})$ is a divisibility sequence. Some authors call this
sequence an elliptic divisibility sequence. In this paper, in order to
distinguish this sequence from the classical elliptic divisibility sequences
studied by Ward, we call the sequence $(D_{nP})$ the \emph{elliptic denominator
sequence} associated to the elliptic curve $E$ and the point $P$.

Comparing equations \eqref{eqn EllDen} and \eqref{eqn DivPoly} we expect a close
relation between $\psi_n(P)$ and $D_{nP}$. In particular, for any prime $p$ we have that
\begin{equation}
    \label{eqn padicvalue}
\nu_p(x(nP))= \nu_p(A_{nP})-2\nu_p(D_{nP}) = \nu_p(\phi_n(P))-2\nu_p(\psi_n(P)), 
\end{equation}
where $\nu_p$ is the $p$-adic valuation on $\QQ$ and
$x(nP)$ is the $x$ coordinate of $nP$.

From construction of division polynomials we know that if $p\nmid D_p$ then $\nu_p(\psi_n(P))\geq 0$ and $\nu_p(\phi_n(P))\geq 0$. Now Theorem \ref{ayad} tells us that if $P$ reduces to a non-singular
point and if $P$ modulo $p$ is different from $\infty$ (i.e. $p \nmid D_P$), then $\nu_p(\psi_n(P))\nu_p(\phi_n(P))=0$. Under these conditions if $\nu_p(x(nP))\geq 0$ then
by \eqref{eqn padicvalue} and the fact that
$A_{nP}$ and $D_{nP}$ are coprime to each other, we have
$\nu_p(D_{nP})=\nu_p(\psi_n(P))=0$. Similarly, if $\nu_p(x(nP))<0$ then
$\nu_p(D_{nP})=\nu_p(\psi_n(P))=-{1\over 2}\nu_p(x(nP))$.

%When $P$ reduces to infinity modulo $p$, 
%then clearly $\nu_p(D_{nP}) \neq \nu_p(\psi_n(P))$ for
%all $n$, since
%they are not equal for $n=1$ (recall, $\psi_1(P)=1$ while $p|D_P$). However, a simple
%a computation using formal groups (see \cite{}) shows that 
%\[ \nu_p(\psi_n(P))=\nu_p(n)-(n^2 -1)\nu_p(D_P) = \nu_p(n)+\nu_p(D_P)-n^2\nu_p(D_P) = 
%\nu_p(D_{nP}) - n^2\nu_p(D_P). \]
%\[\nu_p(D_{nP})=n^2\nu_p(D_P)+\nu_p(\psi_n(P)).\]

Therefore, we have the following proposition.

\begin{prop} 
    \label{first-proposition} 
    Let $E/\QQ$ be an elliptic curve over the rationals given by equation \eqref{WE},
    and assume that $a_i \in \ZZ$.  Furthermore, let $P \in E(\QQ)$ be a point of
    infinite order such that $P \not\equiv \infty \pmod p$ and let $(D_{nP})$ be the elliptic denominator
    sequence associated to $E$ and $P$.    
    %Suppose the
    %assumptions of Theorem \ref{ayad} hold for $E$, $P$, and all primes of bad
    %reduction $p$.  
Then for a prime $p$ if $P \pmod p$ is non-singular, we have 
    $$\nu_p(D_{nP})=\nu_p({\psi}_n(P)).$$
    
%    $$D_{nP}=D_P^{n^2} |\psi_n(P)|.$$ 
%    More generally if $P \pmod p$ is non-singular, then
%    $$\nu_p(D_{nP})=n^2\nu_p(D_P)+\nu_p({\psi}_n(P)).$$
\end{prop}
\begin{rem}\label{mark}
(a)  
One can drop the condition $P \not\equiv \infty \pmod p$ in the previous proposition and prove a stronger result for an scaled version of $\psi_n(P)$. Let $$\hat{\psi}_n(P):=D_P^{n^2} {\psi}_n(P).$$ Then if $P \pmod p$ is non-singular for all primes  $p$, we have
    $$D_{nP}=|{\hat\psi}_n(P)|.$$(See \cite{Ayad} ). For a  proof of this fact (in more general case of elliptic nets) see Proposition \ref{second-proposition}.

\noindent (b) Formulas for explicit valuations of $\psi_n(P)$ at primes $p$ (of good or bad reduction) are given in \cite{Stange3}. Also in \cite{SS} the sign of $\psi_n(P)$ is computed
explicitly.
\end{rem}

 In \cite{Stange}, Stange generalized the concept of an 
elliptic sequence to an
$n$-dimensional array, called an elliptic net. In this paper we give a generalization of Ayad's theorem for net polynomials.
\begin{defn} 
    Let $A$ be a free Abelian group of finite rank, and $R$ be an
    integral domain. Let ${\mathbf 0}$ and $0$ be the additive identity elements of $A$ and $R$ respectively. An {elliptic net} is any map $W:A \rightarrow R$ for which
    $W({\mathbf 0}) = 0$, and that satisfies 
    \begin{multline} \label{net recurrence}
        W(\p+\q+\s)W(\p-\q)W(\rr+\s)W(\rr) \\ + W(\q+\rr+\s)W(\q-\rr)W(\p+\s)W(\p) \\ +
        W(\rr+\p+\s)W(\rr-\p)W(\q+\s)W(\q) = 0, 
    \end{multline} 
    for all $\p,\q,\rr,\s \in A.$ We identify the {rank} of $W$ with the rank of $A$.  
\end{defn}

Note that if $A=\ZZ$ and $W:A\rightarrow R$ is an elliptic net, then by setting $\p=m$, $\q=n$,
$\rr=1$, and $\s={ 0}$ in \eqref{net recurrence}, and noting that $W$ is an odd function, we get that $W(n)$ satisfies equation
\eqref{eds1}, hence $(W(n))$ is an elliptic sequence. Therefore elliptic nets are a generalization
of elliptic sequences.

We can relate elliptic nets to elliptic curves in the following way.
For an arbitrary field $K$, let
$$ S=K[x_1,y_1,\cdots, x_r, y_r], $$ 
and 
consider the polynomial ring
$$
    \mathcal{R}_r=
    K[x_i, y_i]_{1\leq i \leq r}[(x_i-x_j)^{-1}]_{1\leq i<j\leq r}/\langle f(x_i, y_i)\rangle _{1\leq i \leq r},
$$ 
where $f$ is the defining polynomial \eqref{WE} for $E$.
Let $\mathbf{P}=(P_1,P_2,\ldots,P_r) \in E(K)^r$ and 
$\mathbf{v} = (v_{1}, v_{2}, \dots, v_{r}) \in
\ZZ^{r}$. From \cite[Section 4]{Stange} follows that there exist ``polynomials"
$\Psi_{\mathbf{v}}, \Phi_{\mathbf{v}}, \Omegabar_{\mathbf{v}}\in \mathcal{R}_r$
such that $\Psi_\vv$ (as a function of $\vv \in \ZZ^r$) is an elliptic net and
\begin{equation} 
    \label{netp} 
    \mathbf{v} \cdot \mathbf{P}=v_{1}P_{1} + v_{2}P_{2} + \dots + v_{r}P_{r} =
    \Big(\frac{\Phi_{\mathbf{v}}(\mathbf{P})}{\Psi^{2}_{\mathbf{v}}(\mathbf{P})},
    \frac{\Omegabar_{\mathbf{v}}(\mathbf{P})}{\Psi^{3}_{\mathbf{v}}(\mathbf{P})}\Big).
\end{equation} 
The ``polynomial'' $\Psi_{\mathbf{v}}$ is called the \emph{$\mathbf{v}$-th net polynomial} 
associated to $E$. Also, the function $\vv \mapsto \Psi_{\vv}(\mathbf{P})$ is called
\emph{the elliptic net} associated to $E$ and $\mathbf{P}$. In \cite{Stange},
Stange also proves that when $r>1$, then 
we can compute $\Psi_\vv$ using the recurrence relation
\eqref{net recurrence} and the initial values
$\Psi_\vv$ for $\vv=\ee_i$, $\vv=2\ee_i$, $\vv=\ee_i + \ee_j$ and $\vv=2\ee_i+\ee_j$, where $\{\ee_1,\ee_2,\ldots,\ee_r\}$
is the standard basis for $\ZZ^r$. (For $r=1$ the recurrence \eqref{dpr} shows that $\psi_n$ is uniquely determined by $\psi_1$, $\psi_2$, $\psi_3$, and $\psi_4$.)
Note that the initial values of $\Psi_\vv$ are defined as follows:
\begin{align} 
    \begin{aligned}
    \label{Psi init}
    &\Psi_{\mathbf{e_{i}}}= 1,~ 
    \Psi_{\mathbf{2e_{i}}}= 2y_{i} + a_{1}x_{i} + a_{3},~ 
    \Psi_{\mathbf{e_{i}+e_{j}}} = 1,\\
    &\Psi_{2\mathbf{e_{i}}+\mathbf{e_{j}}} = 
    2x_{i} + x_{j} -\Big(\frac{y_{j}-y_{i}}{x_{j}-x_{i}}\Big)^{2} -a_{1} \Big(
    \frac{y_{j}-y_{i}}{x_{j}-x_{i}} \Big) +a_{2}.
\end{aligned}
\end{align} 
The above initial conditions define the $\vv$-th net polynomials of rank $r>1$ for any
elliptic curves completely.  We refer the reader to Theorem 2.5, Lemma 2.6, and Theorem 2.8 
of \cite{Stange} for the details of how this can be done.

In this paper, we prove the following generalization of Theorem \ref{ayad} for net polynomials.  
Let $K$, $\nu$, $\OO$, and $\pp$ be defined as before.
\begin{thm} \label{first-theorem} 
    Let $E/K$ be an elliptic curve
    defined by the polynomial \eqref{WE} with $a_i \in \OO$ for
    $i=1, 2, 3, 4, 6$. 
    Let $\mathbf{P} = (P_{1}, P_{2}, \dots, P_{r}) \in E(K)^{r}$ be such that
    $P_{i} \not\equiv \infty \pmod \pp$, for $1\leq i \leq r$, and
    $P_i\pm P_j \not\equiv \infty \pmod \pp$, for $1\leq i <j \leq r$.
    Then the following are equivalent:
    \begin{enumerate}[(a)] 
        \item \label{property 1} 
            There exists $1 \leq i \leq r$, such that 
            $$
                \nu(\Psi_{2\ee_{i}}(\mathbf{P})) > 0 ~~~~ {\rm and} 
            ~~~~ \nu(\Psi_{3\ee_{i}}(\mathbf{P})) > 0.
            $$ 
        \item \label{property  2}
            There exists $1 \leq i \leq r$ such that for all $n \geq 2$ we have 
            $$\nu(\Psi_{n\ee_{i}}(\mathbf{P})) > 0.$$ 
        \item \label{property 3} 
            There exists $\mathbf{v} \in \ZZ^{r}$ and $1 \leq i \leq r$ such that 
            $$
            \nu(\Psi_{\vv}(\mathbf{P})) > 0 ~~~~ {\rm and}
            ~~~~ \nu(\Psi_{\vv+\ee_{i}}(\mathbf{P})) > 0.$$ 
        \item \label{property 4} 
            There exists $\vv \in \ZZ^{r}$ such that 
            $$
            \nu(\Psi_{\mathbf{v}}(\mathbf{P})) > 0 ~~~~ {\rm and} ~~~~
            \nu(\Phi_{\mathbf{v}}(\mathbf{P})) > 0.
            $$ 
        \item \label{property 5}
            There exists $1 \leq i \leq r$ such that $P_{i} \pmod {\pp}$ is singular.
    \end{enumerate} 
\end{thm}

To prove this, we first need to show that
$\nu(\Psi_{\vv}(\mathbf{P})) \geq 0$ in the cases we are dealing with. This result is of independent interest, so we record it in the following proposition.

\begin{prop} \label{valuation-prop}
    Let $E/K$ be an elliptic curve defined by the polynomial \eqref{WE}
    with $a_{i} \in \OO$ for $i = 1,2,3,4,6$, and let $\mathbf{P} =
    (P_{1}, P_{2}, \dots, P_{r}) \in E(K)^{r}$.  
    When $r=1$, assume that $P_1 \not \equiv \infty \pmod \pp$.
    When $r>1$, then assume that for all $1 \leq i<j \leq r$ we have
    $P_i \not \equiv \infty \pmod \pp$ and
    $P_i \pm P_j \not \equiv \infty \pmod \pp$.
    Then for all $\vv \in \ZZ^r$ we have
    \[ \nu(\Psi_\vv(\mathbf{P})) \geq 0, \]
    hence $\Psi_\vv(\mathbf{P}) \in \OO$.
\end{prop} 

Next we specialize to the case that $E$ is defined over $\QQ$. Let
$E/\mathbb{Q}$ be an elliptic curve, and let ${\mathbf{P}}=(P_1,P_2,\dots,P_r) \in E(\QQ)^r$
be $r$ linearly independent points in $E(\QQ)$.
For $\mathbf{v}=(v_1,v_2,\cdots,v_r) \in \ZZ^r$, let
${\mathbf{v}}\cdot \mathbf{P}=v_1P_1+\cdots+v_r P_r$. 
We denote the \emph{elliptic denominator net} associated to $E$ and $P$ by $(D_{\mathbf{v}\cdot \mathbf{P}})$, where
$D_{\mathbf{v}\cdot \mathbf{P}}$ is the denominator of $\mathbf{v} \cdot \mathbf{P}$.
More precisely,
\begin{equation}
\label{Denom}
    \mathbf{v} \cdot \mathbf{P}= v_1P_1+v_2P_2+\dots+v_rP_r=
    \left(\frac{A_{\mathbf{v} \cdot
    \mathbf{P}}}{D_{\mathbf{v} \cdot \mathbf{P}}^2}, \frac{B_{\mathbf{v} \cdot
    \mathbf{P}}}{D_{\mathbf{v} \cdot  \mathbf{P}}^3}  \right).
\end{equation}
We are interested
in the relation between the element $D_{\mathbf{v}
\cdot \mathbf{P}}$ of the elliptic denominator net, and the value of
the $\mathbf{v}$-th net polynomial
$\Psi_{\mathbf{v}}$ at $\mathbf{P}$. 
An immediate corollary of Theorem \ref{first-theorem} is that for
all but finitely many primes $p$ we have 
\[ \nu_p(D_{\vv \cdot \mathbf{P}}) = \nu_p(\Psi_\vv(\mathbf{P})), \]
where $\nu_p$ is the $p$-adic valuation. 
We extend this result, however similar to  Remark  \ref{mark} (a), we need to multiply $\Psi_{\mathbf{v}}$ at $\mathbf{P}$ with a quadratic form to obtain an equivalent net polynomial ${\hat{\Psi}}_{\mathbf{v}}$. 
%(see Section ?? for exact definition of ${\hat{\Psi}}_{\mathbf{v}}$). 
More precisely, by using notation \eqref{Denom}, let
\begin{equation}
\label{FVP}
F_\vv(\mathbf{P}) = \prod_{1 \leq i \leq j \leq r} A_{ij}^{v_iv_j}, 
\end{equation}
where 
$$ A_{ii}=D_{\ee_i \cdot \mathbf{P}}=D_{P_i},~~{\rm and}~~ A_{ij}=\frac{D_{P_i+P_j}}{D_{P_i} D_{P_j}}~~{\rm for}~~i\neq j.$$
Then $F(\mathbf{P}):\ZZ^r \rightarrow K^\times$ defined by
$\vv \mapsto F_\vv(\mathbf{P})$ is a quadratic form.
Define
\[ \Psihat_\vv(\mathbf{P}) = F_\vv(\mathbf{P})\Psi_\vv(\mathbf{P}), \]
for all $\vv \in \ZZ^r$.
Then $\Psihat(\mathbf{P})$ is an elliptic net that is scale equivalent
to $\Psi(\mathbf{P})$ (see Section \ref{sec2} for more explanation).
Furthermore, notice that
\[ \Psihat_{\ee_i}(\mathbf{P})=F_{\ee_i}(\mathbf{P})\Psi_{\ee_i}(\mathbf{P})=A_{ii}=D_{\ee_i \cdot \mathbf{P}}, \]
and
\[ \Psihat_{\ee_i+\ee_j}(\mathbf{P})=F_{\ee_i+\ee_j}(\mathbf{P})\Psi_{\ee_i+\ee_j}
(\mathbf{P})=A_{ii}A_{jj}A_{ij}=D_{P_i+P_j}=D_{(\ee_i+\ee_j) \cdot \mathbf{P}}. \]

We will prove the following generalization of Proposition \ref{first-proposition}.
\begin{prop} \label{second-proposition}
    Let $E/\QQ$ be an elliptic net defined by polynomial \eqref{WE}
    with  $a_i\in \ZZ$ for $i=1, 2, 3, 4, 6$. Let
    $\mathbf{P}=(P_1,\ldots,P_r) \in E(\QQ)^r$ be an $r$-tuple consisting of $r$ linearly independent points in $E(\QQ)$.
    Let $p$ be a prime so that $P_i \pmod p$ is non-singular for $1\leq i \leq r$.
    Then
    \[ 
        \nu_p(D_{\vv \cdot \mathbf{P}})=\nu_p(\Psihat_\vv(\mathbf{P})),
    \]
    for all $\vv \in \ZZ^r$.
    In particular, if for all primes $p$ and all integers $1 \leq i \leq r$
    we have that $P_i \pmod p$ is nonsingular, then
    \[ 
        D_{\vv \cdot \mathbf{P}} = |\Psihat_\vv(\mathbf{P})|. 
    \]

\end{prop}

Section \ref{sec5} includes proofs of Propositions \ref{valuation-prop}, \ref{second-proposition}, and Theorem \ref{first-theorem}. Also see Examples \ref{first-example} and \ref{second-example} for concrete descriptions of Proposition \ref{second-proposition}.

To prove Theorem \ref{first-theorem}, we need to study the behaviour of zeros of
an elliptic net $W:\ZZ^r \rightarrow \kk$, where $\kk$ is an arbitrary field. Recall that for the values of rank $1$ elliptic nets (i.e. elliptic
sequences), we have the concept of {\em rank of apparition}. More precisely, for any 
elliptic sequence $(W_n)$ we say that a natural number $\rho$ is a rank of apparition if $W_{\rho}=0$ and
$W_m \neq 0$ for any $m | \rho$.  
We say a sequence has {\em a unique rank of apparition} $\rho~ (>1)$ if 
$W_k = 0 $ if and only if $\rho | k$.
Motivated by this definition, we say an elliptic net $W:\ZZ^r \rightarrow \kk$ has a 
{\em unique rank of apparition with respect to the standard basis}
if each sequence $(W(n\ee_1)),~ (W(n\ee_2)),~\ldots,~(W(n\ee_r))$  has a unique rank of apparition.
In general, it is convenient to have a definition that works for a free finitely generated
Abelian group $A$, rather than $\ZZ^r$. 

\begin{defn} 
    \label{apparition} 
    Let $W: A \rightarrow \kk$ be an elliptic net of rank $r$. 
    Let $\mathscr{B}=\{\bb_1, \bb_2, \dots, \bb_r\}$ be a basis for $A$. We say
    that $W$ has a {\em unique rank of apparition with respect to 
        $\mathscr{B}$} if there exists an
    $r$-tuple $(\rho_1, \rho_2, \dots, \rho_r)$ of positive integers with
    $\rho_{i} > 1$ for $1 \leq i \leq r$, such that 
    $$W(n\bb_i)=0 \iff \rho_i \mid n,$$ 
    for all $1\leq i \leq r$.
\end{defn}
Note that an elliptic sequence $(W_n)$ has a unique rank of apparition if its corresponding net $n \mapsto W_n$ has a
unique rank of apparition with respect to $\{1\}$.

We remark here another possible generalization of a
unique rank of apparition of a sequence. Namely,
for a sequence $W:\ZZ \rightarrow \kk$, having a unique rank of apparition is the
same as $\Lambda = \{ v \in \ZZ \st W(v)=0\}$ being a subgroup of $\ZZ$.
Therefore, a natural generalization of the concept of unique rank of apparition to
elliptic nets $W:A \rightarrow \kk$ is that $\Lambda = W^{-1}(0)=\{\vv \in A \st W(\vv)=0\}$ to be
a subgroup of $A$. 
The following theorem shows that our concept of unique rank of apparition implies
that $\Lambda$ is a subgroup of $A$.
\begin{thm} 
    \label{second-theorem} 
    Let $W:A \rightarrow \kk$ be an elliptic net, and let
    $\mathscr{B}=\{\bb_1, \dots, \bb_r\}$ be a basis for $A$. Assume that $W$ has a
    unique {rank of apparition with respect to $\mathscr{B}$}.  Let  
    $$\Lambda = W^{-1}(0)=\{\vv \in A \st W(\vv) = 0 \}$$
    be the zero set of $W$. Then $\Lambda$ is a full rank subgroup of $A$. 
\end{thm}
\noindent We prove Theorem \ref{second-theorem} in Section \ref{sec3}.

The proof of Theorem \ref{first-theorem} comes as a combination of Theorems
\ref{ayad}, \ref{second-theorem} and the following theorem.

\begin{thm}[\bf Ward] 
    \label{Ward6.2}
    Let $(W_n)$ be an elliptic sequence.
    A necessary and sufficient condition that $(W_n)$ does not have a unique 
    rank of apparition  is that $W_3=W_4=0$.
\end{thm}
The proof of Theorem \ref{Ward6.2} is analogous to \cite[Theorem 6.2]{Ward}
where the case of an integer elliptic sequence modulo $p$ has been considered. 
%We prove
%Theorem \ref{second-theorem} in Section ??.  Section \ref{sec3} includes a proof
%of Theorem \ref{first-theorem}.

Here we describe another application of Theorem \ref{second-theorem}. Let $W_n=\hat{\psi}_n(P)$ as defined in Remark \ref{mark} (a). Then Proposition \ref{first-proposition} tells
us that in many cases, we can think of $W_n$ as the denominator of the point $nP$
for some elliptic curve $E/\QQ$ and some point $P \in E(\QQ)$. Now let $p$ be a prime
of good reduction and let $n_p$ be the order of the point $P$ in $E(\FF_p)$, where $\FF_p$ is the finite field of $p$ elements. Then 
we have that $(n_p+k)P \equiv kP \pmod p$.  Therefore, it is tempting to assume
that $W_{n_p+k} \equiv W_k \pmod p$. More generally, let $W:\ZZ \rightarrow \kk$ be
an elliptic sequence with $\rho$ the unique rank of apparition of $W$.
Then, one may speculate that $W_{\rho+k} = W_k$.
This in fact is not true. However, in \cite{Ward} the following is proved.
\begin{thm}[\bf Ward's Symmetry Theorem] 
    \label{Ward sym} 
    Let $(W_n)$ be an elliptic sequence and assume $W_2W_3 \neq 0$.
    Let $\rho>1$ be the unique rank of apparition of $W$.
    Then there exists $a, b \in \kk$ such that
    \begin{equation*} 
        W_{m\rho+n}  = a^{m^2} b^{mn} W_{n}
    \end{equation*} 
    for all $m,n \in \ZZ$.
\end{thm} 
\noindent See Theorem 9.2 of \cite{Ward} for a proof and Theorem 8.2 of \cite{Ward} for some properties of elements $a$ and $b$, when $\kk=\FF_p$. 
Note that the proofs also work for any field $\kk$.

The following theorem gives a generalization of Theorem \ref{Ward sym}.

\begin{thm}\label{Ward-generalized} 
    Let $W:A \rightarrow \kk$ be an elliptic net with the property that $\Lambda=W^{-1}(0)$ is a subgroup of $A$
    and assume $|A/\Lambda| \geq 4$. Then, there exist well defined functions
    $\xi :\Lambda \rightarrow \kk^\times$ and $\chi:\Lambda \times A \rightarrow \kk^\times$ such that
$$W(\bm{\lambda}+\vv)=\xi(\bm{\lambda})\chi(\bm{\lambda},\vv)W(\vv)~ for ~ all ~\bm{\lambda} \in \Lambda~
            and~ all ~\vv \in A,$$
and the functions $\xi$ and $\chi$ satisfy the following properties:
\begin{enumerate}[(i)]
        \item $\chi$ is bilinear,
        \item $\chi(\bm{\lambda}_1,\bm{\lambda}_2)=\chi(\bm{\lambda}_2,\bm{\lambda}_1)$,
        \item $\xi(\bm{\lambda}_1+\bm{\lambda}_2)=\xi(\bm{\lambda}_1)\xi(\bm{\lambda}_2)\chi(\bm{\lambda}_1,\bm{\lambda}_2)$, 
        \item $\xi(-\bm{\lambda})=\xi(\bm{\lambda})$, and
        \item $\xi(\bm{\lambda})^2 = \chi(\bm{\lambda},\bm{\lambda})$.
    \end{enumerate}
    Furthermore, the functions $\chi(\bm{\lambda},\p)$ and $\xi(\bm{\lambda})$, are defined by 
    $$
        \begin{array}{cccc} 
            \delta: & \Lambda \times (A \setminus \Lambda) & \longrightarrow & \kk^\times \\ 
            & (\bm{\lambda}, \p) & \longmapsto & \frac{W(\bm{\lambda}+\p)}{W(\p)},  
        \end{array}
    $$ 
    and relations
    $$
        \begin{array}{cccc} 
            \chi: & \Lambda \times A & \longrightarrow & \kk^\times \\ 
            & (\bm{\lambda}, \p) & \longmapsto & \frac{\delta(\bm{\lambda}, \p+\vv)}{\delta(\bm{\lambda}, \vv)}, 
        \end{array}
    $$ 
    where $\vv$ is any element of $A$ with $\vv,\vv+\p \notin \Lambda$, 
    and 
    $$
        \begin{array}{cccc} 
            \xi: & \Lambda & \longrightarrow & \kk^\times \\ 
            & \bm{\lambda} & \longmapsto & \frac{\delta(\bm{\lambda},\vv)}{\chi(\bm{\lambda},\vv)}, 
        \end{array}
    $$ 
    for any $\vv \in A\setminus \Lambda$.  
\end{thm}
Note that under conditions of Theorem \ref{Ward sym}, by considering $\bm{\lambda}=m\rho$ and $\vv=n$ in the previous theorem  and applying the bilinearity of $\chi$ and Corollary \ref{nsquare}, we obtain
$$W(m\rho+n)=\xi(m\rho) \chi(m\rho, n)W(n)=\xi(\rho)^{m^2} \chi(\rho, 1)^{mn}W(n).$$
Thus, by letting $a=\xi(\rho)$ and $b=\chi(\rho, 1)$ we have the assertion of Theorem \ref{Ward sym}.

We remark here that in \cite{Stange1}, Stange relates some of the functions given in Theorem \ref{Ward-generalized} 
to the Tate pairing on $E$. Furthermore, special cases of the above formula does show up in her
thesis. However to the best of our knowledge, the statement of the above theorem is new.

Given the properties of $\chi$ and $\xi$, for any $r \in \NN$, any
$\bm{\lambda}_1,\bm{\lambda}_2,\ldots,\bm{\bm{\lambda}}_r \in \Lambda$, and any $n_1,n_2,\ldots,n_r \in \ZZ$,
we get that
    \begin{equation}
    \label{W-formula}
        W\left(\left(\sum_{i=1}^{r}n_{i}\bm{\lambda}_{i}\right)+\vv \right) =
        \left(\prod_{i=1}^r \xi(\bm{\lambda}_i)^{n_i^2} \chi(\bm{\lambda}_i,\vv)^{n_i}
            \left(\prod_{j=1}^{i-1} \chi(\bm{\lambda}_i,\bm{\lambda}_j)^{n_in_j} \right)
            \right) W(\vv).
    \end{equation} 
As a simple corollary of the above identity, we have the following periodicity result. 
\begin{cor}
\label{cor-13}
Let $W:A \rightarrow \FF_q$ be an
elliptic net, and let $\Lambda=W^{-1}(0)$. Assume that $\Lambda$ is a subgroup of $A$  and
$|A/\Lambda|\geq 4$. 
%Assume that $\vv_1 \equiv \vv_2 \pmod{(q-1)\Lambda}$. 
Then
$W(\vv_1)=W(\vv_2)$ if $\vv_1 \equiv \vv_2 \pmod{ (q-1)}$.
\end{cor}
We can also employ \eqref{W-formula} in computing elliptic nets with values in finite fields (See Example \ref{third-example} for a description). Section \ref{sec4} is dedicated to proofs of Theorem \ref{Ward-generalized} and
its corollaries.

\section{Review of Elliptic Nets}
\label{sec2}
We will collect some basic facts about elliptic nets in this section for 
sake of completion. 
Recall that for a free Abelian group $A$ and an integral domain $R$,
we defined an elliptic net to be any map $W: A \rightarrow R$ with $W(\mathbf{0})=0$ and
\begin{multline*} 
    W(\p+\q+\s)W(\p-\q)W(\rr+\s)W(\rr) \\ + W(\q+\rr+\s)W(\q-\rr)W(\p+\s)W(\p) \\ +
    W(\rr+\p+\s)W(\rr-\p)W(\q+\s)W(\q) = 0, 
\end{multline*}
for all $\p,\q,\rr,$ and $\s \in A$.
Also recall that the rank of an elliptic net is defined to be the rank of
its domain $A$.

\begin{lem}
    \label{lem prelim}
    Let $W:A \rightarrow R$ be an elliptic net.
    \begin{enumerate}[(a)]
        \item For any integral domain $R'$ and any morphism $\pi:R \rightarrow R'$, the
            function $\pi \circ W:A \rightarrow R'$ is an elliptic net,
        \item For any subgroup $A' \subset A$, the function
             $W|_{A'} :A' \rightarrow R$ is an elliptic net.
        \item For any $\vv \in A$ we have $W(-\vv)=-W(\vv)$,
    \end{enumerate}
\end{lem}
\begin{proof}
    To prove the first two parts of this lemma, note that both
    $W|_{A'}$ and $\pi \circ W$ satisfy the elliptic net recurrence
    \eqref{net recurrence}.
    To prove $W(-\vv)=-W(\vv)$, observe that if  $W(\vv)=W(-\vv)=0$, then we are
    done.  Otherwise, assume without loss of generality that $W(\vv) \neq 0$. Then by
    setting $\p=\q=\vv$ and $\rr=\s=\mathbf{0}$ in \eqref{net recurrence} we have
    \[ W(\vv)^3(W(\vv)+W(-\vv))=0. \]
    Since $R$ is an integral domain, we get $W(-\vv)=-W(\vv)$.
\end{proof}

We have already remarked that the values of an elliptic net of rank $1$ form
an elliptic sequence. Let $W:A \rightarrow R$ be
any elliptic net, and let $\vv \in A$. Then by part (b) of the
above lemma, $W|_{ \vv\ZZ } : \ZZ \rightarrow R$
is an elliptic net of rank $1$.
Also, note that if $R$ is an integral domain and $\kk = \Frac(R)$, the fraction field of $R$, then 
$i:R \rightarrow \kk$ is injective. Therefore $i \circ W: A \rightarrow \kk$
is an elliptic net, and $(i \circ W)^{-1}(0)=W^{-1}(0)$.
Therefore we are not losing any generality in Theorems \ref{second-theorem}
and \ref{Ward-generalized} by focusing on elliptic nets 
having entries in a field.

Next we are interested in relating elliptic nets with linear combination of points on elliptic curves. In order to do this we review some results of \cite{Stange} on net polynomials. 

For a complex lattice $\Lambda \subset \CC$, let 
            $\sigma:\CC \rightarrow \CC$ be the Weierstrass
            $\sigma$ function
            \begin{align*}
                \sigma(z)=\sigma(z;\Lambda) = z \prod_{w \in \Lambda, w \neq 0}
                \left( 1 - {z \over w}\right) e^{ {z\over w} + {1\over 2} ({z\over w})^2} .
            \end{align*}
            Fix an $r$-tuple ${\mathbf z} = (z_1,z_2,\ldots,z_r) \in \CC^r$
            with $z_i \not \in \Lambda$ and $z_i + z_j \not \in \Lambda$.
            For an $r$-tuple $\vv=(v_1,v_2,\ldots,v_r) \in \ZZ^r$ define
            \begin{align*}
                \Omega_\vv({\mathbf z}) = \Omega_\vv(\zz ; \Lambda) = (-1)^{\sum_{1\leq i \leq j \leq r} v_i v_j +1}{\sigma(v_1z_1+v_2z_2+\cdots+v_rz_r) \over
                    \left( \prod_{i=1}^r \sigma(z_i)^{2v_i^2 - \sum_{j=1}^r v_iv_j} \right) 
                    \left( \prod_{1 \leq i < j \leq r} \sigma(z_i+z_j)^{v_iv_j} \right)}.
            \end{align*}
            \begin{thm}[{\bf Stange}]
            The function
            \begin{align*} 
                \begin{array}{cccc} 
                    \Omega = \Omega(\zz;\Lambda): & \ZZ^r & \longrightarrow & \CC \\ 
                    & \vv & \longmapsto & \Omega_\vv(\zz),  
                \end{array}
            \end{align*}
            is an elliptic net.
            \end{thm}
            \begin{proof}
            See \cite[Theorem 3.7]{Stange}.
            \end{proof}
        
        Now let $E/\CC$ be an elliptic curve, and let 
            $\Lambda_E$ be the lattice corresponding to $E$.
            Let
            ${\mathbf P}=(P_1,\ldots,P_r) \in E(\CC)^r$ with
            $P_i, P_i + P_j \neq \infty$ and
            let ${\mathbf z}=(z_1,\ldots,z_r) \in \CC^r$ be such
            that $z_i$ maps to $P_i$ under the uniformization map
            \[ \CC \rightarrow \CC/\Lambda_E \simeq E(\CC). \]
            Then the function
            \begin{align*} 
                \begin{array}{cccc} 
                    \Psi(\mathbf{P};E): & \ZZ^r & \longrightarrow & \CC \\ 
                    & \vv & \longmapsto & \Omega_\vv(\zz),  
                \end{array}
            \end{align*}
            is an elliptic net with values in $\mathbb{C}$. We call $\Psi({\bf P};E)$ the \emph{elliptic net associated to $E( {\rm over~} \CC)$ and $\mathbf{P}$}.
       
            Let $S^\univ=\ZZ[\alpha_1,\alpha_2,\alpha_3,\alpha_4,\alpha_6]$, and for
            any positive integer $r$ let
            \[ 
                \RU_{r}^{\univ} = S^{\univ}[x_i,y_i]_{1 \leq i \leq r}[(x_i-x_j)^{-1}]_{1 \leq i < j \leq r}/\langle f(x_i,y_i) \rangle _{1 \leq i \leq r},
            \]
            where $f(x, y)$ is given by \eqref{WE}.
            Then for every elliptic curve $E/K$ defined by the polynomial
            \eqref{WE}, and $\mathbf{P} \in E(K)^r$ with $P_i, P_i \pm P_j \neq \infty$, 
            we can find a morphism
            \[ \pi=\pi_{\mathbf{P};E} : \RU_{r}^{\univ} \rightarrow K \]
            so that
            $\pi(\alpha_i)=a_i$, and $(\pi(x_i),\pi(y_i))=P_i$.
            The following result is proved in \cite[section 4]{Stange}. 
            \begin{thm}[{\bf Stange}]
            For each $\vv \in \ZZ^r$,
            there is $\Psi_{\vv}^{\univ} \in \RU_{r}^{\univ}$ so that
            $\Psi^\univ:\vv \mapsto \Psi_{\vv}^{\univ}$ is an elliptic net, and
            for any elliptic curve $E/\CC$ and $\mathbf{P} \in E(\CC)^r$
            with $P_i, P_i \pm P_j \neq \infty$ we have
            \[ \pi_{\mathbf{P};E}\circ \Psi^\univ = \Psi(\mathbf{P};E). \]
        \end{thm}
Let $\RU_r^\univ$, $S^\univ$, and  
            $E/K$ be as before.
            Then, there exists a map $\pi_E : S^\univ \rightarrow K$, so that
            $\pi_E(\alpha_i)=a_i$.
            This induces a map
            \[ (\pi_E)_* : \RU_r^\univ \rightarrow
            K[x_i,y_i]_{1 \leq i \leq r}[(x_i-x_j)^{-1}]_{1 \leq i < j \leq r}/\langle f(x_i,y_i) \rangle _{1 \leq i \leq r}. \]
            Then part (a) of lemma \ref{lem prelim} shows that 
            $\Psi : \vv \mapsto (\pi_E)_* (\Psi_\vv^\univ)$ defines an elliptic net with values in 
            \[ \RU_r := K[x_i,y_i]_{1 \leq i \leq r}[(x_i-x_j)^{-1}]_{1 \leq i < j \leq r}/\langle f(x_i,y_i) \rangle _{1 \leq i \leq r}.  \]
            We call $\Psi_\vv \in \RU_r$ the \emph{$\vv$-th net polynomial} associated to $E$.
            \remove{In this paper, we use $\Psi_\vv$ instead of $(\pi_E)_*(\Psi_\vv^\univ)$, when
            $E/K$ is clear.}
            Now let 
            ${\mathbf{P}} \in E(K)^r$ with $P_i, P_i \pm P_j \neq \infty$. 
            Then by part (a) of Lemma \ref{lem prelim}, 
            $\Psi(\mathbf{P};E):\vv \mapsto \Psi_\vv(\mathbf{P})$ is an elliptic net with values in $K$.
We call $\Psi({\bf P};E)$ the \emph{elliptic net associated to $E( {\rm over~} K)$ and $\mathbf{P}$}.  

Here we note that $\Psi_{n\ee_i}(\mathbf{P})=\psi_n(P_i)$. Moreover, we remark that for $E/K$ defined by the polynomial \eqref{WE}, we
can compute $\Psi_\vv$ explicitly. In fact for 
$\vv \in \{ \ee_i, 2\ee_i, \ee_i+\ee_j, 2\ee_i+\ee_j \st i \neq j \}$, the exact
values of $\Psi_\vv$ are given by \eqref{Psi init}.
Furthermore, as we pointed out in the introduction, Theorem 2.5, Lemma 2.6, 
and Theorem 2.8 of \cite{Stange} prove that these initial conditions
are sufficient for computing $\Psi_\vv$ for any $\vv \in \ZZ^r$.
\begin{exam}
    If we let $(\p,\q,\rr,\s)=(\ee_i+\ee_j, \ee_i\pm \ee_k, -\ee_i, -\ee_i)$, then from \eqref{net recurrence} we get that
    \[ \Psi_{\ee_i+\ee_j \pm \ee_k} \Psi_{\ee_j\mp \ee_k} \Psi_{-2\ee_i}\Psi_{-\ee_i}+\Psi_{-\ee_i \pm \ee_k} \Psi_{2\ee_i\pm \ee_k} \Psi_{\ee_j}\Psi_{\ee_i+\ee_j}+\Psi_{-\ee_i\pm \ee_k} \Psi_{-2\ee_i- \ee_j} \Psi_{\pm\ee_k}\Psi_{\ee_i\pm \ee_k} = 0 .\]
We note that in \cite[Theorem 2.5]{Stange} it is shown that the terms $\Psi_{\ee_i - \ee_j}$, and $\Psi_{2\ee_{i} - \ee_{j}}$ can be computed explicitly in terms of $\Psi_{\vv}$ for $\vv \in \{ \ee_i, 2\ee_i, \ee_i+\ee_j, 2\ee_i+\ee_j \st i \neq j \}$. In particular, setting $(\p, \q, \rr, \s) = (\ee_i, \ee_j, \mathbf{0}, \ee_i + \ee_j)$ gives
    \[\Psi_{\ee_i - \ee_j} = \Psi_{\ee_i + 2\ee_j} - \Psi_{2\ee_i + \ee_j}.\]
Similarly, taking $(\p,\q,\rr,\s) = (-\ee_i + \ee_j, \ee_j, \ee_i, \ee_i)$, we have
    \[\Psi_{2\ee_i -\ee_j} = \psi_{2\ee_i}\psi_{2\ee_j} - \psi_{2\ee_i + \ee_j}\psi_{\ee_i - \ee_j}^2. \]
Thus $\Psi_{\ee_i+\ee_j \pm \ee_k}$ can be computed using $\Psi_{\vv}$ for $\vv \in \{ \ee_i, 2\ee_i, \ee_i+\ee_j, 2\ee_i+\ee_j \st i \neq j \}.$
\end{exam}

We are interested in relating $\Psi_\vv(\mathbf{P})$ to the denominators of 
linear combinations of points on an elliptic curve. To do this, recall that
for any $E/K$ given by the polynomial \eqref{WE}, 
$\mathbf{P} \in E(K)^r$, and $\vv \in \ZZ^r$ we can find rational
functions (by repeated use of doubling and addition formulas for elliptic curves) 
$X_\vv, Y_\vv \in \Frac(\RU_r)$, the fraction field of $\RU_r$, such that
\[ \vv \cdot \mathbf{P} = v_1P_1+\cdots + v_r P_r = (X_\vv(\mathbf{P}),Y_\vv(\mathbf{P})). \]
The following lemma gives an explicit representation for $X_\vv$ in terms of net polynomials.
\begin{lem}
    \label{numerator formula}
    Let $E/K$ be an elliptic net, and let $\mathbf{P} \in E(K)^r$ be such that
    $P_i \neq \infty$ and $P_i \pm P_j \neq \infty$. Then
    For any $\vv \in \ZZ^r$, there is $\Phi_\vv \in \RU_r$ such that
    \[ X_{\vv} = {\Phi_\vv \over \Psi_\vv^2}. \]
    In particular for any $1 \leq i \leq r$ we have
    \[ \Phi_\vv (\mathbf{P}) = \Psi_\vv^2(\mathbf{P}) x(P_i) - \Psi_{\vv + \ee_i}(\mathbf{P})\Psi_{\vv-\ee_i}(\mathbf{P}). \]
\end{lem}
\begin{proof}
    In \cite[Lemma 4.2]{Stange}, it is proved that for any
    $\vv,\uu \in \ZZ^r$ we have
    \[ \Psi_\vv^2 \Psi_\uu^2(X_\vv - X_\uu)=-\Psi_{\vv+\uu}\Psi_{\vv-\uu}. \]
    If we let $\uu=\ee_i$, then
    $X_\uu (\mathbf{P})= x(P_i)$. Thus we have 
    \[ (\Psi_\vv^2 X_\vv)(\mathbf{P}) = \Psi_\vv^2(\mathbf{P}) x(P_i) - \Psi_{\vv + \ee_i}(\mathbf{P})\Psi_{\vv-\ee_i}(\mathbf{P}), \]
    which gives us the desired result.
\end{proof}

\begin{defn}
Let $B$ and $C$ be Abelian groups  written additively. Furthermore, assume that $C$
is $2$-torsion free. Then a function $F:B \rightarrow C$ 
is a quadratic form if
\begin{align}
    \label{paral}
    F(x+y)+F(x-y)=2F(x)+2F(y), 
\end{align}
for all $x,y \in B$. 
\end{defn}
Equation \eqref{paral} is sometimes called the {\em parallelogram law}.
\begin{exam}
    \begin{enumerate}[(a)]
        \item Let $a_i, c_{ij} \in \QQ$ and consider $F:\ZZ^r \rightarrow \QQ$ defined by
            \[ F(v_1,v_2,\ldots,v_r) = \sum_{i=1}^r a_i v_i^2 + \sum_{1 \leq i < j \leq r} c_{ij}v_iv_j. \]
            Then we can check that $F$ satisfies the parallelogram law \eqref{paral}.
        \item Let $p_i, q_{ij} \in \QQ^\times$. Then
            the function $G:\ZZ^r \rightarrow \QQ^\times$ defined by
            \[ G(v_1,v_2,\ldots,v_r) = \prod_{i=1}^r p_i^{ v_i^2} \cdot \prod_{1 \leq i < j \leq r} q_{ij}^{v_iv_j} \]
            is a quadratic form.
        \item Let $F_1,F_2 : B \rightarrow C$ be two quadratic forms. Then
            their difference, $F_1-F_2$, is again a quadratic form.
    \end{enumerate}
\end{exam}
The main reason we are interested in quadratic forms is the following result.
\begin{prop}
    Let $K$ be a field and let $W:A \rightarrow K$ be an elliptic net.
    Let $F:A \rightarrow K^\times$ be a quadratic form. Then
    \begin{equation}
        \begin{array}{cccc} 
            W^F: & A & \longrightarrow & K \\
            & \vv & \longmapsto & W(\vv)F(\vv)
        \end{array}
    \end{equation}
    is an elliptic net.
\end{prop}
\begin{proof}
    See \cite[Proposition 6.1]{Stange}.
\end{proof}
\begin{defn}
We say that two elliptic nets $W$ and $W^\prime$ are scale equivalent, if there is a quadratic form $F: A \rightarrow K$ such that $W^\prime=W^F$. 
\end{defn}
%Let $E/\QQ$ defined by the polynomial \eqref{WE}, and
%    assume $a_i$'s are all integer. Let $\Delta_E$ be the discriminant of $E$. Let $P\neq \infty$ be a point in $E(\QQ)$ and consider the function 
%    $$\lambda_p(P)=\max\left\{-{1\over 2}\nu_p(x(P)),0\right\}+\frac{1}{12} \nu_p(\Delta_E).$$
Let $\lambda_p$ be the \emph{(local) N\'{e}ron height function on $E$ associated to the prime $p$}. (See \cite[Chapter VI, Theorem{1.1}]{AECII} for properties of $\lambda_p$.)
An important property of N\'{e}ron height is that it satisfies the \emph{quasi-parallelogram law}.

\begin{lem} 
\label{quasi}
Assume that $P$, $Q \in E(\QQ)$ are two points such that $P$, $Q$, $P\pm Q\neq \infty$. Then we have  
$$\lambda_p(P+Q)+\lambda_p(P-Q)=2\lambda_p(P)+2\lambda_p(Q)+\nu_p(x(P)-x(Q))-\frac{1}{6}\nu_p(\Delta_E).$$
\end{lem}
\begin{proof}
See \cite[Page 476, Exercise 6.3]{AECII}.
\end{proof}

\begin{lem}
    \label{diff quad}
    Let $E/\QQ$ defined by the polynomial \eqref{WE}, and
    assume $a_i$'s are all integers. Let $\Delta_E$ be the discriminant of $E$. Let
    $\mathbf{P}=(P_1,\ldots,P_r) \in E(\QQ)^r$ be an $r$-tuple consisting of $r$ 
    linearly independent points on $E(\QQ)$.
   % Let $p$ be a prime so that $P_i \pmod p$ is non-singular
    %for all $1 \leq i \leq r$.
    Define
    \begin{align*}
        \eps(\vv) = 
        \begin{cases}
            \lambda_p(\vv \cdot \mathbf{P})-\frac{1}{12} \nu_p(\Delta_E)-\nu_p(\Psi_\vv(\mathbf{P})) & \mbox{ if $\vv \neq \mathbf{0}$,} \\
            0 & \mbox{ otherwise.} \\
        \end{cases}
    \end{align*}
    Then $\eps$ is a quadratic form from $\ZZ^r$ to $\ZZ$.
\end{lem}
\begin{proof}
%    First we note that if $\vv\cdot \mathbf{P}=\infty$ .....
        From Lemma \ref{numerator formula}, we know that
    \begin{align}
        \label{par Psi}       \nu_p(\Psi_{\vv+\ww}(\mathbf{P}))+\nu_p(\Psi_{\vv-\ww}(\mathbf{P})) = 
        2\nu_p(\Psi_\vv(\mathbf{P}))+2\nu_p(\Psi_\ww(\mathbf{P}))+
        \nu_p(X_\vv(\mathbf{P})-X_\ww(\mathbf{P})). 
    \end{align}
%    On the other hand, note that for any $P \in E(\QQ)\setminus \{\infty\}$
%    \[ \nu_p(D_P)=\max\left(-{1\over 2}\nu_p(x(P)),0\right). \]
%    Let $\lambdat(P)=\max(-{1\over 2}\nu_p(x(P)),0)$ and
%    note that 
%    \[ 
%        \eps(\vv)=\nu_p(D_{\vv \cdot \mathbf{P}})-\nu_p(\Psi_\vv(\mathbf{P})) = \lambdat(\vv \cdot \mathbf{P})-\nu_p(\Psi_\vv(\mathbf{P})). 
%    \]
%    If $P$ reduces to a non-singular point, we have
%    \[ \lambdat(P)=\lambda(P)-{1\over 12}\nu_p(\Delta_E), \]
%    where $\lambda$ is the {\em local N\'eron height} function on $E$ 
%    associated to $p$
%    (see \cite[Chapter VI]{AECII}).
%    Using properties of local N\'eron height, we get that
%    for any $P,Q \in E(\QQ)$ we have
%    \begin{align*}
%        \lambda(P+Q)+\lambda(P-Q)=2\lambda(P)+2\lambda(Q)+\nu_p(x(P)-x(Q))-{1\over 6}\nu_p(\Delta_E),
%    \end{align*}
%    or equivalently
%    \begin{align}
%        \label{loc ht par}
%        \lambdat(P+Q)+\lambdat(P-Q)=2\lambdat(P)+2\lambdat(Q)+\nu_p(x(P)-x(Q))
%    \end{align}
%    (see \cite[Exercise 6.3]{AECII} or \cite{Panda}).
Now assume that $\vv, \ww, \vv\pm \ww \neq {\mathbf 0}$. Then
substituting $\vv \cdot \mathbf{P}$ and $\ww \cdot \mathbf{P}$ in Lemma \ref{quasi},
    %\eqref{loc ht par}, 
    we get
    \begin{multline}
        \label{par den}
        \lambda_p(\vv\cdot \mathbf{P}+\ww \cdot \mathbf{P})+
        \lambda_p(\vv\cdot \mathbf{P}-\ww \cdot \mathbf{P}) = 
        2\lambda_p(\vv \cdot \mathbf{P})+
        2\lambda_p(\ww \cdot \mathbf{P})+
        \nu_p(X_\vv(\mathbf{P})-X_\ww(\mathbf{P}))-\frac{1}{6} \nu_p(\Delta_E).
    \end{multline}
    Subtracting \eqref{par Psi} from \eqref{par den} we have
    \begin{align}
    \label{p-law}
        \eps(\vv+\ww)+\eps(\vv-\ww)=2\eps(\vv)+2\eps(\ww),
    \end{align}
where $\vv, \ww, \vv\pm \ww \neq \mathbf{0}$. The identity \eqref{p-law} also holds if $\vv$ or $\ww=\mathbf{0}$. So to complete the proof it is enough to show that $\eps(2\vv)=4\eps(\vv)$. In order to establish this we add copies of \eqref{p-law} for $(\vv, \ww)=(4\uu, \uu), (3\uu, \uu), (3\uu, \uu), (2\uu, \uu)$ to obtain
\begin{equation}
\label{first}
\eps(5\uu)+\eps(\uu) = 2\eps(3\uu) + 8\eps(\uu)
%2\eps(3\uu)+2\eps(4\uu)+2\eps(2\uu)=2\eps(4\uu)+2\eps(3\uu)+8\eps(\uu).
\end{equation}
Also letting $(\vv, \ww)=(3\uu, 2\uu)$ in \eqref{p-law} yields
\begin{equation}
\label{second}
\eps(5\uu)+\eps(\uu)=2\eps(3\uu)+2\eps(2\uu).
\end{equation}
Now subtracting \eqref{second} from \eqref{first} gives $\eps(2\uu)=4\eps(\uu).$
Thus $\eps$ is a quadratic form as desired.
\end{proof}

\section{Proof of Theorem \ref{second-theorem}} 
{\label{sec3}}
Let $\kk$ be any field and assume that $W:A \rightarrow \kk$ is an elliptic net of rank $r$.
Theorem \ref{second-theorem} claims that if $W$ has a unique
rank of apparition then $\Lambda=W^{-1}(0)$ will be a subgroup
of $A$. The goal of this section is to prove this claim.

Throughout this section assume that $W$ has a 
unique rank of apparition and let $\Bas=\{\bb_1,\bb_2,\ldots,\bb_r\}$ be a
basis for $A$ such that $W$ has a unique rank of apparition with respect
to $\Bas$.
Therefore, there exists $(\rho_1,\rho_2,\ldots,\rho_r) \in \ZZ^r$ with
$\rho_i > 1$ for $1 \leq i \leq r$ such that 
$W(n\bb_i)=0$ if and only if $n | \rho_i$.

Let
$A_i$ be the subgroup of $A$ generated by $\{\bb_1,\bb_2,\ldots,\bb_i\}$ for
$1 \leq i \leq r$ and let
\[\Lambda_i = \Lambda \cap A_i = \{ \vv \in A_i \st W(\vv)=0 \}.\]
Note that $\Lambda_i$ is the zero set of the elliptic net $W|_{A_i}:A_i \rightarrow \kk$.
By induction on $i$, 
we will prove that $\Lambda_i$ is a subgroup of $A_i$.
Note that the base case, $i=1$, is true by definition 
of unique rank of apparition.

We will prove the inductive step by proving three lemmas. 
\begin{lem} \label{lem2}
    Let $n \in \ZZ$, and let $1 \leq i \leq r$. If $\rho_i \mid n$, then we have 
    $$W(\vv+n\bb_{i}) = 0 \Longleftrightarrow \vv \in \Lambda.$$ 
\end{lem}

\begin{proof} 
    First let $ \vv\in \Lambda$. Taking $\p=\vv$, $\q = -n\bb_{i},$ $\rr =
    \bb_{i},$ and $\s = 2n\bb_{i}$ in \eqref{net recurrence} yields
    \begin{equation} \label{zero id 1} 
        W(\vv+n\bb_{i})^{2}W((2n+1)\bb_{i})W(\bb_{i}) = 0.
    \end{equation}
    Note that since $\rho_i \mid n$ and $\rho_{i}>1$, 
    we have $\rho_i \nmid (2n+1)$ and so $W((2n+1)\bb_{i})\neq 0$. 
    Thus, from \eqref{zero id 1}, we have $W(\vv+n\bb_{i}) = 0$ for all $\vv \in \Lambda$.

    Conversely assume that $\vv \notin \Lambda$. Then taking $\p=\vv$, 
    $\q = n\bb_{i},$ $\rr = \bb_{i},$ and $\s= \mathbf{0}$ in \eqref{net recurrence} yields
    \begin{equation} \label{zero id 2} 
        W(\vv+n\bb_{i})W(\vv-n\bb_{i})W(\bb_{i})^{2} + W((n+1)\bb_{i})W((n-1)\bb_{i})W(\vv)^{2} = 0.
    \end{equation}
    Since $\vv \notin \Lambda$ and $\rho_{i}\mid n$, we have
    $W((n+1)\bb_{i})W((n-1)\bb_{i})W(\vv)^{2} \neq 0$. It therefore follows,
    from \eqref{zero id 2}, that $W(\vv+n\bb_{i}) \neq 0$ for all $\vv \notin \Lambda$.  
\end{proof} 
The following is a
straightforward consequence of Lemma \ref{lem2}.  

\begin{cor} We have
    $$\{n_1\bb_1+n_2\bb_2+\dots+n_r\bb_r \st \rho_i \mid n_i ~\mbox{for}~1\leq i \leq r\} \subseteq
\Lambda.$$ 
\end{cor}

\begin{lem} \label{lem3} 
    Suppose that for a fixed $i>1$ we have that
    $\Lambda_{i-1}$ is a subgroup of $A$. 
    Then for all $\vv \in \Lambda_{i-1}$, we have 
    $$W(\vv + n\bb_{i}) = 0 \Longleftrightarrow \rho_{i} \mid n.$$ 
\end{lem}

\begin{proof} 
    Choose $\vv \in \Lambda_{i-1}$.
    Since $\vv \in \Lambda_{i-1} \subset \Lambda$, it
    follows from Lemma \ref{lem2} that if $\rho_{i}\mid n$ then 
    $W(\vv+n\bb_{i}) = 0$. 
    Conversely, let $\rho_{i} \nmid n$, taking $\p=\vv$, $\q = n\bb_{i},$
    $\rr \in A_{i-1} \setminus \Lambda_{i-1},$ and
    $\s = \mathbf{0}$ in \eqref{net recurrence} yields
    \begin{equation} \label{star} 
        W(\vv+n\bb_{i})W(\vv-n\bb_{i})W(\rr)^{2} +
        W(\rr+\vv)W(\rr-\vv)W(n\bb_{i})^{2} = 0.  
    \end{equation} 
    Since $\vv \in \Lambda_{i-1},$  $\rr \in A_{i-1} \setminus \Lambda_{i-1}$, and
    $\Lambda_{i-1}$ is a subgroup, it follows that $\vv \pm \rr \in A_{i-1}
    \setminus \Lambda_{i-1}$, hence $W(\vv\pm \rr) \neq 0$. It therefore follows
    from \eqref{star} that $W(\vv+n\bb_{i}) \neq 0$.
\end{proof}

\begin{lem} \label{lem4}
    Suppose that $\Lambda_{i-1}$ is a subgroup of $A$ for a fixed $i>1$ and $\rho_{i} > 2$.
    Let $\uu, \vv \in \Lambda_{i}$ such that $\uu = \uu_{0} + n\bb_{i},$ 
    and $\vv = \vv_{0} + n\bb_{i}$ for $\uu_{0}, \vv_{0} \in A_{i-1}$. 
    Then $\uu-\vv=\uu_0-\vv_0 \in \Lambda_{i-1}.$ 
\end{lem}

\begin{proof}
% Assume that $\Lambda_{i-1}$ is a lattice, and $\rho_{i} > 2$. Throughout we
% let $p = p_{0} + nb_{i}, ~ q = q_{0}+nb_{i}$ with $p,q \in \Lambda_{i}$, and
% $p_{0}, q_{0} \in A_{i-1}$.

Setting $\p = \uu_{0} + n\bb_{i},$ 
$\q = \vv_{0} + n\bb_{i}$, $\rr = m\bb_{i},$ and $\s = -2n\bb_{i}$ in 
\eqref{net recurrence} gives 
\begin{equation} \label{pplusq}
W(\uu_{0} + \vv_{0})W(\uu_{0}-\vv_{0})W((2n-m)\bb_{i})W(m\bb_{i}) = 0.  
\end{equation}
Since $\rho_{i} > 2$, we have $W(b_i),~W(2b_i)\neq 0$.  So we can choose $m \in
\{1,2\}$ such that 
$$W((2n-m)\bb_{i})W(m\bb_{i}) \neq 0.$$ 
Thus from \eqref{pplusq}
we conclude that $W(\uu_{0} + \vv_{0})W(\uu_{0}-\vv_{0}) = 0$. 
%We aim to prove $W(p_{0} + q_{0})=W(p_{0}-q_{0}) = 0$. In order to deduce
%this, without loss of generality,
Now if $W(\uu_{0}-\vv_{0}) = 0$ we are done.  Otherwise we assume that
$W(\uu_{0}-\vv_{0}) \neq 0$, hence $W(\uu_{0} + \vv_{0}) = 0$, and show that this gives
a contradiction. 

Setting $\p = \uu_{0} + n\bb_{i},$ $\q = \vv_{0} + n\bb_{i}$, $\rr = \bb_{i}$, 
and $\s=\mathbf{0}$ in
\eqref{net recurrence} gives 
$$W(\uu_{0}+\vv_{0} + 2n\bb_{i})W(\uu_{0}-\vv_{0})W(\bb_{i})^{2} = 0,$$ 
hence $W(\uu_{0}+\vv_{0} + 2n\bb_{i}) = 0$
(recall that $W(\uu_{0}-\vv_{0}) \neq 0$). Since $\uu_0+\vv_0\in \Lambda_{i-1}$ it
follows from Lemma \ref{lem3} that $\rho_{i}\mid 2n$. Now we consider two
cases.
 
Case 1: If $\rho_{i} \mid n$, then since 
$\uu = \uu_{0} +n\bb_{i}, \vv = \vv_{0}+n\bb_{i} \in \Lambda_{i}$, 
it follows from Lemma \ref{lem2} that
$\uu_{0}, \vv_{0} \in \Lambda_{i-1}$, hence $\uu_{0} -\vv_{0} \in \Lambda_{i-1}$,
contradicting our assumption that $W(\uu_{0} - \vv_{0}) \neq 0.$ 

Case 2: If $\rho_{i} \nmid n$, then $W(\uu_{0}+\vv_{0}+n\bb_{i}) \neq 0$ by Lemma
\ref{lem3}. Setting $\p = \uu_{0} + n\bb_{i}$, $\q = \vv_{0} + n\bb_{i}$, $\rr=\bb_{i},$
and $\s = -n\bb_{i}$ in \eqref{net recurrence} gives
$$W(\uu_{0}+\vv_{0}+n\bb_{i})W(\uu_{0}-\vv_{0})W((n-1)\bb_{i})W(\bb_{i}) = 0,$$ 
hence
$W((n-1)\bb_{i}) = 0$ and so $\rho_i \mid n-1$. Similarly by setting 
$\p = \uu_{0} + n\bb_{i},$ $\q = \vv_{0} + n\bb_{i}$, $\rr=-\bb_{i},$ and
$\s = -n\bb_{i}$ in \eqref{net recurrence}
we find that $W((n+1)\bb_{i}) = 0$ and so $\rho_i \mid n+1$. Since $\rho_i\mid
n-1$ and $\rho_i\mid n+1$, we have $\rho_{i} = 2$. This is a contradiction.
\end{proof}

We are ready to prove our main result on zeros of an elliptic net.

\begin{proof}[Proof of Theorem \ref{second-theorem}] 
    We proceed by induction on $i$. Note that $\Lambda_{1}$ is a subgroup of $\bb_1\ZZ$, since $W(n\bb_{1}) = 0$ 
    if and only if $\rho_{1}|n$.
    Assume that $\Lambda_{i-1}$ is a subgroup.
    We want to prove that $\Lambda_i$ is a subgroup, that is
    for any $\uu,\vv \in \Lambda_i$ that $\uu-\vv \in \Lambda_i$.
    We will prove this by contradiction, so assume that $\uu - \vv \not \in \Lambda_i$.
    Let $\uu=\uu_0+n\bb_i,~\vv=\vv_0+m\bb_i \in \Lambda_{i}$, 
    where $\uu_0, \vv_0\in A_{i-1}$.
    It follows from \eqref{net recurrence}, for $\p=\uu$, $\q=\vv$, $\rr=\uu+\w$, and $\s=-2\uu$,
    that  $W(\uu-\vv)^2W(\uu-\w)W(\uu+\w) = 0.$ Since $W(\uu-\vv)\neq 0$ and  
    $\uu=\uu_0+n\bb_i$, we conclude that  
    \begin{equation} \label{zero id 3} 
        W(\uu_0+n\bb_{i}-\w)W(\uu_0+n\bb_{i}+\w)=0  
    \end{equation} 
    for any $\w \in A_i$.  We claim that \eqref{zero id 3} implies
    that $\rho_{i} \mid n$.
    %Setting $s = -2p$ (respectively $-2q$), and replacing $r$ with $p+r$
    %(respectively $q+r$), we find $$W(p+r)W(p-r) = 0,$$ $$W(q+r)W(q-r) = 0.$$ Let
    %$p = p_{0} + nb_{i}$ with $p_{0} \in A_{i-1}$, and 
    To show this assume otherwise that $\rho_{i} \nmid n$. Then, since 
    $\uu = \uu_{0}+n\bb_{i} \in \Lambda_{i}$ it follows from Lemma \ref{lem3} that
    $\uu_0\not\in \Lambda_{i-1}$. We consider two cases.

    Case 1: If $\rho_{i} > 2$, then setting $\w = \uu_0$ in \eqref{zero id 3} yields
    $$W(2\uu_{0}+n\bb_{i})W(n\bb_{i}) = 0.$$ 
    Then we have that $W(2\uu_{0}+n\bb_{i}) =0$ since $\rho_i \nmid n$. Since
    $W(\uu_{0}+n\bb_{i}) = W(2\uu_{0}+n\bb_{i}) = 0$, it follows from 
    Lemma \ref{lem4} that $\uu_{0} \in \Lambda_{i-1}$. This is a contradiction. 
    % However, since $W(p_{0}+nb_{i}) = 0$, it then follows from Lemma \ref{lem3}
    % that $\rho_{i}\mid n$. 
    %This is a contradiction.
     
    Case 2: If $\rho_{i} = 2$, then setting $\w = \bb_{i}$ in \eqref{zero id 3} yields
    $$W(\uu_{0} + (n+1)\bb_{i})W(\uu_{0} + (n-1)\bb_{i}) = 0,$$ from which it follows that
    $\uu_{0} \in \Lambda_{i-1}$ (since both $n-1$ and $n+1$ are even). This is a
    contradiction. 
    % Now since $W(p_{0}+nb_{i}) = 0$ and $p_0\in \Lambda_{i-1}$, it then follows
    % from Lemma \ref{lem3} that $\rho_{i}\mid n$, 
    %which is a contradiction. 

    In either case, the assumption $\rho_i \nmid n$ leads to a contradiction. Thus, we have 
    $\uu = \uu_{0} + n\bb_{i}$ with $\uu_{0} \in \Lambda_{i-1}$, and $\rho_{i} \mid n$.
    Similarly we have $\vv = \vv_{0} + m\bb_{i}$, with $\vv_{0} \in \Lambda_{i-1}$ and
    $\rho_{i}|m$. Then, $\uu-\vv = \uu_{0} - \vv_{0} + (n-m)\bb_{i}$ with 
    $\uu_{0} - \vv_{0} \in \Lambda_{i-1}$, and $\rho_{i}\mid (n-m)$. 
    Thus it follows from Lemma \ref{lem3}
    that $W(\uu-\vv) = 0$. This is a contradiction as we assumed that $W(\uu-\vv)\neq 0$.  
     
    Since the assumption $\uu-\vv\not\in \Lambda_i$ leads to a contradiction, we
    conclude that $\uu-\vv\in \Lambda_i$ and so $\Lambda_i$ is a subgroup of $A$.  
\end{proof}

\section{Proofs of Theorem \ref{Ward-generalized} and Corollary \ref{cor-13}}
\label{sec4}

Theorem \ref{second-theorem} shows that for a given elliptic net $W:A \rightarrow \kk$,
in favorable conditions, if $W(\bm{\lambda}_1)=W(\bm{\lambda}_2)=0$ then $W(\bm{\lambda}_1+\bm{\lambda}_2)=0$.
In this section we study the relation between  $W(\vv+\bm{\lambda})$ and $W(\vv)$ when $W(\bm{\lambda})=0$
but $W(\vv)$ is non-zero.
Throughout this section we assume that $\Lambda=W^{-1}(0)$ is a subgroup of
$A$. We also assume that $|A/\Lambda| \geq 4$.
The results of this section generalizes Theorem \ref{Ward sym}
to Elliptic nets.
In order to do this, we first define the auxiliary function 
$$
    \begin{array}{cccc} 
        \delta: & \Lambda \times (A \setminus \Lambda) & \longrightarrow & \kk^\times \\ 
              & (\bm{\lambda}, \vv) & \longmapsto & \frac{W(\bm{\lambda}+\vv)}{W(\vv)}, 
    \end{array}
$$ 
and explore the properties of $\delta$. Notice that for $\bm{\lambda} \in \Lambda$ and
$\vv \notin \Lambda$ we get that $\delta(\bm{\lambda},\vv) \neq 0$.
We have the following lemma.

\begin{lem} \label{sigeqn} 
    For all $\bm{\lambda} \in \Lambda$, and 
    $\mathbf{a},\mathbf{b},\mathbf{c},\mathbf{d} \in A \backslash  \Lambda$ 
    with $\mathbf{a}+\mathbf{b} = \mathbf{c}+\mathbf{d}$, we have 
    $$ \delta(\bm{\lambda},\mathbf{a}) \delta(\bm{\lambda}, \mathbf{b}) = 
    \delta(\bm{\lambda}, \mathbf{c}) \delta(\bm{\lambda}, \mathbf{d}).$$

\end{lem}

\begin{proof} 
    Assume that $\p+\s, \p, \q+\s, \q \notin \Lambda$. Then, setting 
    $\rr = \bm{\lambda}$ in \eqref{net recurrence} gives 
    \begin{equation*}
        W(\bm{\lambda}+\q+\s)W(\bm{\lambda}-\q)W(\p+\s)W(-\p) = W(\bm{\lambda}+\p+\s)W(\bm{\lambda} - \p)W(\q+\s)W(-\q).  
    \end{equation*} 
    Since $\p+\s, \p, \q+\s,\q \notin \Lambda$ we have $W(\p+\s)W(\p)W(\q+\s)W(\q) \neq 0$, hence
    $$
        \frac{W(\bm{\lambda}+\q+\s)W(\bm{\lambda}-\q)}{W(\q+\s)W(-\q)} = 
        \frac{W(\bm{\lambda}+\p+\s)W(\bm{\lambda}-\p)}{W(\p+\s)W(-\p)}.
    $$ 
    Thus 
    $$
        \delta(\bm{\lambda}, \q+\s) \delta(\bm{\lambda}, -\q) = \delta(\bm{\lambda}, \p+\s)\delta(\bm{\lambda}, -\p).
    $$ 
    Taking 
    $$ 
        \mathbf{a}=\q+\s,~
        \mathbf{b}=-\q,~ 
        \mathbf{c}=\p+\s,~{\rm and}~ 
        \mathbf{d}=-\p,
    $$ 
    yields the result.  
\end{proof}

Note that if $\vv,\p_1,\p_2 \in A$ and $\p_1,\p_2,\vv+\p_1,\vv+\p_2 \not \in \Lambda$,
then
$$\delta(\bm{\lambda},\vv+\p_1)\delta(\bm{\lambda},\p_2)=\delta(\bm{\lambda},\vv+\p_2)\delta(\bm{\lambda},\p_1).$$
Since $\delta$ is nonzero, we get
\begin{equation}
    \label{eqn well defined}
    {\delta(\bm{\lambda},\vv+\p_1) \over \delta(\bm{\lambda},\p_1)} = 
    {\delta(\bm{\lambda},\vv+\p_2) \over \delta(\bm{\lambda},\p_2)}. 
\end{equation}
Since we are assuming that $|A/\Lambda|\geq 4$, we get that
for any $\vv \in A$ there is an
an element $\p \in A$ so that $\p$ and $\vv+\p$ are in $A \setminus \Lambda$.
In light of this observation, we define the function $\chi$ by
\begin{equation}
    \label{eqn chi}
    \begin{array}{cccc} 
        \chi: & \Lambda \times A & \longrightarrow & \kk^\times \\
        & (\bm{\lambda}, \vv) & \longmapsto & \frac{\delta(\bm{\lambda}, \vv+\p)}{\delta(\bm{\lambda}, \p)},
    \end{array}
\end{equation}
for any choice of $\p$ with $\p, \vv+\p \notin \Lambda$.
Equation \eqref{eqn well defined} shows that this definition is independent
of the choice of $\p$. Furthermore, note that $\delta$ is non-zero, so $\chi$ maps to $\kk^\times$.

We now show that $\chi$ is a bilinear map.

\begin{lem}\label{chi lem} 
    Let $W:A \rightarrow \kk$ be an elliptic net, and $\Lambda=W^{-1}(0)$ be a subgroup
    of $A$ such that $|A/\Lambda| \geq 4$. Let $\chi:\Lambda\times A \rightarrow \kk^\times$ 
    be defined as before.
    Then for  $\bm{\lambda}, \bm{\lambda}_{1}, \bm{\lambda}_{2} \in \Lambda,$ and  
    $\vv, \vv_{1}, \vv_{2} \in A$, we have the following:
    \begin{enumerate}[(i)] 
        \item $\chi(\bm{\lambda}, \vv_{1} + \vv_{2}) = \chi(\bm{\lambda}, \vv_{1})\chi(\bm{\lambda}, \vv_2).$ 
        \item $\chi(\bm{\lambda}_{1} + \bm{\lambda}_{2}, \vv) = \chi(\bm{\lambda}_{1}, \vv)\chi(\bm{\lambda}_{2}, \vv).$ 
        \item $\chi(\bm{\lambda}_{1}, \bm{\lambda}_{2}) = \chi(\bm{\lambda}_{2}, \bm{\lambda}_{1}).$ 
        \item\label{chi inverse} $\chi(\bm{\lambda}, -\vv) = \chi(\bm{\lambda}, \vv)^{-1}.$ 
    \end{enumerate} 
\end{lem}

\begin{proof} 
    First we note that if $|A/\Lambda| \geq 4$, then for any choice of $\vv_1,\vv_2 \in A$,
    we can find $\p \in A$ so that $\p,\p+\vv_2,$ and $\p+\vv_1+\vv_2$ are not
    in $\Lambda$. In particular, by pigeonhole principle, we can find $\overline{\uu} \in A/\Lambda$
    so that the image of $\mathbf{0},\vv_2$ and $\vv_1+\vv_2$ will miss $\overline{\uu}$ in
    $A/\Lambda$. Letting $\p$ be any element in $A$ that reduces to $-\overline{\uu}$ we get
    the desired result.
    Given this $\p$ we have,
    \begin{eqnarray*}
        \chi(\bm{\lambda}, \vv_{1})\chi(\bm{\lambda}, \vv_{2}) & = & 
        \frac{\delta(\bm{\lambda}, \vv_{1}+\vv_{2}+\p)}{\delta(\bm{\lambda}, \vv_{2}+\p)} \frac{\delta(\bm{\lambda}, \vv_{2}+\p)}{\delta(\bm{\lambda}, \p)} \\ 
        & = & \frac{\delta(\bm{\lambda}, \vv_{1}+\vv_{2}+\p)}{\delta(\bm{\lambda}, \p)} \\ 
        & = & \chi(\bm{\lambda}, \vv_{1}+\vv_{2}).
    \end{eqnarray*}
    This proves the first statement.

    For the second statement, we let $\p \in A \backslash \Lambda$ be such that 
    $\vv + \p \notin \Lambda$ (Again, by pigeonhole principle, such an element exists). 
    Since $\Lambda$ is a subgroup of $A$, it follows that 
    $\vv+\p+\bm{\lambda}_{2}, \p+\bm{\lambda}_{2} \notin \Lambda$. 
    Hence, we have 
    \begin{eqnarray*} 
        \chi(\bm{\lambda}_{1}, \vv) \chi(\bm{\lambda}_{2}, \vv) & = & 
        \frac{\delta(\bm{\lambda}_{1}, \vv+\p+\bm{\lambda}_{2}) \delta(\bm{\lambda}_{2}, \vv+\p)}{\delta(\bm{\lambda}_{1}, \p+\bm{\lambda}_{2}) \delta(\bm{\lambda}_{2}, \p)} \\ 
        & = & 
        \frac{W(\vv+\p+\bm{\lambda}_{1}+\bm{\lambda}_{2})W(\p+\bm{\lambda}_{2})W(\vv+\p+\bm{\lambda}_{2})W(\p)}{W(\vv+\p+\bm{\lambda}_{2})W(\p+\bm{\lambda}_{1}+\bm{\lambda}_{2})W(\vv+\p)W(\p+\bm{\lambda}_{2})}
        \\ & = & 
        \frac{W(\vv+\p+\bm{\lambda}_{1}+\bm{\lambda}_{2})W(\p)}{W(\vv+\p)W(\p+\bm{\lambda}_{1}+\bm{\lambda}_{2})} \\ 
        & = & \frac{\delta(\bm{\lambda}_{1}+\bm{\lambda}_{2}, \vv+\p)}{\delta(\bm{\lambda}_{1}+\bm{\lambda}_{2}, \p)} \\ 
        & = & \chi(\bm{\lambda}_{1}+\bm{\lambda}_{2}, \vv).  
    \end{eqnarray*}

    For the third statement, taking $\p\in A\setminus \Lambda$, we have 
    $$ \chi(\bm{\lambda}_{1}, \bm{\lambda}_{2}) = 
    \frac{\delta(\bm{\lambda}_{1}, \bm{\lambda}_{2} + \p)}{\delta(\bm{\lambda}_{1}, \p)} = 
    \frac{W(\bm{\lambda}_{1} + \bm{\lambda}_{2} +\p)W(\p)}{W(\bm{\lambda}_{2} + \p)W(\bm{\lambda}_{1} +\p)}  = 
    \frac{\delta(\bm{\lambda}_{2}, \bm{\lambda}_{1} + \p)}{\delta(\bm{\lambda}_{2}, \p)}  = 
    \chi(\bm{\lambda}_{2}, \bm{\lambda}_{1}). $$

    The last statement follows from $(i)$ and the fact that $\chi(\bm{\lambda}, 0) = 1$.
\end{proof}

Note that for $\bm{\lambda} \in \Lambda$ and $\vv \notin \Lambda$ we have 
\[ W(\vv+\bm{\lambda})=\delta(\bm{\lambda},\vv)W(\vv)={\delta(\bm{\lambda},\vv) \over \chi(\bm{\lambda},\vv)} \chi(\bm{\lambda},\vv)W(\vv). \]
We now show that $\delta(\bm{\lambda},\vv)/\chi(\bm{\lambda},\vv)$ is independent of choice of $\vv$. 

\begin{lem} \label{constlem} 
    For all $\vv_{1}, \vv_{2} \in A \backslash \Lambda$ we have 
    \begin{equation*} 
        {\delta(\bm{\lambda},\vv_1) \over \chi(\bm{\lambda},\vv_1)} = {\delta(\bm{\lambda}, \vv_{2}) \over \chi(\bm{\lambda},\vv_2)}.
    \end{equation*}
\end{lem}

\begin{proof} 
    First, if $\vv_{1} + \vv_{2} \notin \Lambda$ we have
    \begin{equation}\label{a identity 1} 
        \frac{\delta(\bm{\lambda}, \vv_{1})}{\chi(\bm{\lambda}, \vv_{1})} = 
        \frac{\delta(\bm{\lambda}, \vv_{1}) \delta(\bm{\lambda}, \vv_{2})}{\delta(\bm{\lambda}, \vv_{1}+\vv_{2})} = 
        \frac{\delta(\bm{\lambda}, \vv_{2})}{\chi(\bm{\lambda}, \vv_{2})}. 
    \end{equation}

    Next, we suppose that $\vv_{1} + \vv_{2} \in \Lambda$. Then, 
    since $|A/\Lambda| \geq 4$, we can find $\p \in A\setminus \Lambda$ such that 
    $\p \not\equiv -\vv_{1}, -\vv_{2} \pmod \Lambda$. Then, we have
    \begin{equation*}
        \vv_{1} +\vv_{2} + \p, 2\vv_{1}+\vv_{2}+\p, \vv_{1}+2\vv_{2}+\p \notin \Lambda.
    \end{equation*}
    It then follows from \eqref{a identity 1}, that
    \begin{equation*}
        {\delta(\bm{\lambda},\vv_{1}) \over \chi(\bm{\lambda},\vv_1)}= 
        {\delta(\bm{\lambda},\vv_{1}+\vv_{2}+\p) \over \chi(\bm{\lambda},\vv_1+\vv_2+\p)} = 
        {\delta(\bm{\lambda},\vv_{2}) \over \chi(\bm{\lambda},\vv_2)}.
    \end{equation*}
\end{proof}
 
Now in light of Lemma \ref{constlem}, we define
\begin{equation}
    \begin{array}{cccc} 
        \xi: & \Lambda & \longrightarrow & \kk^\times \\
        & \bm{\lambda} & \longmapsto & \frac{\delta(\bm{\lambda}, \vv)}{\chi(\bm{\lambda}, \vv)},
    \end{array}
\end{equation}
for any choice of $\vv \in A \setminus \Lambda$. Lemma \ref{constlem} shows that
$\xi$ is a well defined function. 

We are now in a position to give a generalization of Theorem \ref{Ward sym}.
\begin{reptheorem}{Ward-generalized} 
    Let $W:A \rightarrow \kk$ be an elliptic net with the property that $\Lambda=W^{-1}(0)$ is a subgroup of $A$
    and assume $|A/\Lambda| \geq 4$. Then, there exist well defined functions
    $\xi :\Lambda \rightarrow \kk^\times$ and $\chi:\Lambda \times A \rightarrow \kk^\times
    $ such that
$$W(\bm{\lambda}+\vv)=\xi(\bm{\lambda})\chi(\bm{\lambda},\vv)W(\vv)~ for ~ all ~\bm{\lambda} \in \Lambda~
            and~ all ~\vv \in A,$$
and the functions $\xi$ and $\chi$ satisfy the following properties:
    \begin{enumerate}[(i)]
        \item $\chi$ is bilinear,
        \item $\chi(\bm{\lambda}_1,\bm{\lambda}_2)=\chi(\bm{\lambda}_2,\bm{\lambda}_1)$,
        \item $\xi(\bm{\lambda}_1+\bm{\lambda}_2)=\xi(\bm{\lambda}_1)\xi(\bm{\lambda}_2)\chi(\bm{\lambda}_1,\bm{\lambda}_2)$, 
        \item $\xi(-\bm{\lambda})=\xi(\bm{\lambda})$, and
        \item $\xi(\bm{\lambda})^2 = \chi(\bm{\lambda},\bm{\lambda})$.
    \end{enumerate}
\end{reptheorem}

\begin{proof} 
    Recall that we have defined the functions
    $\delta(\bm{\lambda},\vv)={W(\vv+\bm{\lambda}) \over W(\vv)}$,
    $\chi(\bm{\lambda},\vv)=\frac{\delta(\bm{\lambda},\vv+\p)}{\delta(\bm{\lambda},\p)}$,
    $\xi(\bm{\lambda})={\delta(\bm{\lambda},\vv) \over \chi(\bm{\lambda},\vv)}$ 
    for any choice of $\vv, \p \in A$ so that the fractions make sense.
    Note that
    \[ W(\vv+\bm{\lambda})=\delta(\bm{\lambda},\vv)W(\vv) = \xi(\bm{\lambda})\chi(\bm{\lambda},\vv)W(\vv), \]
    for any $\vv \not \in \Lambda$. If $\vv \in \Lambda$ then both sides are $0$.
    Therefore, for any $\vv \in A$ and any $\bm{\lambda} \in \Lambda$ we have
    \begin{equation}
        W(\vv+\bm{\lambda})=\xi(\bm{\lambda})\chi(\bm{\lambda},\vv)W(\vv)
        \label{eqn monodromy}
    \end{equation}
    Furthermore, Lemma \ref{chi lem} shows that $\chi$ is bilinear and 
    $\chi|_{\Lambda \times \Lambda}$ is symmetric.

    Therefore, all we have to do is to show that 
    \begin{equation}
        \label{eqn asum} 
        \xi(\bm{\lambda}_1+\bm{\lambda}_2)=\xi(\bm{\lambda}_1)\xi(\bm{\lambda}_2)\chi(\bm{\lambda}_1,\bm{\lambda}_2), 
    \end{equation}
    that $\xi(-\bm{\lambda})=\xi(\bm{\lambda})$, and 
    \begin{equation}
        \label{eqn asqr}
        \xi(\bm{\lambda})^2 = \chi(\bm{\lambda},\bm{\lambda}). 
    \end{equation}
    Let $\bm{\lambda}_1,\bm{\lambda}_2 \in \Lambda$ and $\vv \notin \Lambda$. 
    Note that by \eqref{eqn monodromy} and (i) we get
    \[ W(\bm{\lambda}_1+\bm{\lambda}_2+\vv)=\xi(\bm{\lambda}_1+\bm{\lambda}_2)\chi(\bm{\lambda}_1+\bm{\lambda}_2,\vv)W(\vv) = \xi(\bm{\lambda}_1+\bm{\lambda}_2)\chi(\bm{\lambda}_1,\vv)\chi(\bm{\lambda}_2,\vv)W(\vv). \]
    On the other hand
    \begin{align*}
        W(\bm{\lambda}_1+(\bm{\lambda}_2+\vv))=& \xi(\bm{\lambda}_1)\chi(\bm{\lambda}_1,\vv+\bm{\lambda}_2)W(\vv+\bm{\lambda}_2) \\
        =& \xi(\bm{\lambda}_1)\xi(\bm{\lambda}_2)\chi(\bm{\lambda}_1,\vv+\bm{\lambda}_2)\chi(\bm{\lambda}_2,\vv)W(\vv) \\
        =& \xi(\bm{\lambda}_1)\xi(\bm{\lambda}_2)\chi(\bm{\lambda}_1,\bm{\lambda}_2)\chi(\bm{\lambda}_1,\vv)\chi(\bm{\lambda}_2,\vv)W(\vv). 
    \end{align*}
    Equating the above two equations for $W(\bm{\lambda}_1+\bm{\lambda_2}+\vv)$ yields
    \[ \xi(\bm{\lambda}_1+\bm{\lambda}_2)\chi(\bm{\lambda}_1,\vv)\chi(\bm{\lambda}_2,\vv) = \xi(\bm{\lambda}_1)\xi(\bm{\lambda}_2)\chi(\bm{\lambda}_1,\bm{\lambda}_2)\chi(\bm{\lambda}_1,\vv)\chi(\bm{\lambda}_2,\vv), \]
    which gives us \eqref{eqn asum}.

    Now note that $\xi(\mathbf{0})=1$, since
    $W(\vv+\mathbf{0})=\xi(\mathbf{0})\chi(\mathbf{0},\vv)W(\vv)=W(\vv)$.
    Similarly, 
    \begin{align*}
        W(-\vv-\bm{\lambda})&=\xi(-\bm{\lambda})\chi(-\bm{\lambda},-\vv)W(-\vv) \\
        &=\xi(-\bm{\lambda})\chi(\bm{\lambda},\vv)W(-\vv) \\
        &=-\xi(-\bm{\lambda})\chi(\bm{\lambda},\vv)W(\vv)
    \end{align*}
    while
    \begin{align*}
        W(-\vv-\bm{\lambda})&=-W(\vv+\bm{\lambda}) \\
        &=-\xi(\bm{\lambda})\chi(\bm{\lambda},\vv)W(\vv)
    \end{align*}
    which implies $\xi(-\bm{\lambda})=\xi(\bm{\lambda})$.
    Therefore
    \[ 1=\xi(\mathbf{0})=\xi(\bm{\lambda}-\bm{\lambda})=\xi(\bm{\lambda})\xi(-\bm{\lambda})\chi(\bm{\lambda},-\bm{\lambda}), \]
    which by employing part (iv) of Lemma \ref{chi lem} results in $\xi(\bm{\lambda})^2=\chi(\bm{\lambda},\bm{\lambda})$. This completes the proof of our
    theorem.
\end{proof}

As an immediate corollary of the above theorem we have
\begin{cor}
\label{nsquare}
    Let $W:A\rightarrow \kk$ be an elliptic net with $\Lambda=W^{-1}(0)$ be a subgroup
    of $A$ and $|A/\Lambda|\geq 4$. Then for all $\bm{\lambda} \in \Lambda$ and $n \in \ZZ$
    we have
    \[ \xi(n\bm{\lambda})=\xi(\bm{\lambda})^{n^2}. \]
\end{cor}
\begin{proof} 
    We already showed that $\xi(\mathbf{0})=1$, so the statement holds for $n=0$.
    It also trivially holds for $n=1$. We proceed by induction.
    Assume the statement is true for some $n \geq 1$. From part (4) of 
Theorem \ref{Ward-generalized} and Lemma \ref{chi lem}, we have 
    $$\xi((n+1)\bm{\lambda}) = 
    \xi(\bm{\lambda})\xi(n\bm{\lambda}) \chi(\bm{\lambda}, n\bm{\lambda}) = 
    \xi(\bm{\lambda})\xi(n\bm{\lambda}) \chi(\bm{\lambda}, \bm{\lambda})^{n}.$$
    From the induction hypothesis and part (v) of Theorem \ref{Ward-generalized}, it follows that
    $$\xi((n+1)\bm{\lambda}) = \xi(\bm{\lambda})^{n^2+1}\xi(\bm{\lambda})^{2n} = \xi(\bm{\lambda})^{(n+1)^2}.$$
    Therefore the statement holds for all $n \geq 0$.
    Finally note that $\xi(-n\bm{\lambda})=\xi(n\bm{\lambda})=\xi(\bm{\lambda})^{n^2}$ from
    part (5) of Theorem \ref{Ward-generalized}.
    Thus the statement holds for all $n \in \ZZ$. 
\end{proof}

Note that Theorem \ref{Ward-generalized} allows us to compute $W:A \rightarrow \kk$
by knowing the values of $W$ on a set of representatives of $A/\Lambda$ and
by computing certain  values of $\chi$ and $\xi$. In particular if $\Lambda$ is a full rank subgroup of $A$,
then we can choose $\bm{\lambda}_1,\bm{\lambda}_2,\ldots,\bm{\lambda}_r$ as a basis of $\Lambda$.
Then
\begin{align*}
    W\left(\left(\sum_{i=1}^r n_i\bm{\lambda}_i\right)+\vv\right) &= \xi\left(\sum_{i=1}^r n_i\bm{\lambda}_i\right) \chi\left(\sum_{i=1}^r n_i\bm{\lambda}_i,\vv\right) W(\vv) \\
    &=\xi\left(\sum_{i=1}^r n_i\bm{\lambda}_i\right) \prod_{i=1}^r \chi\left(\bm{\lambda}_i,\vv\right)^{n_i}W(\vv)
\end{align*}
and 
\begin{align*}
    \xi\left(\sum_{i=1}^r n_i\bm{\lambda}_i\right) &= \prod_{i=1}^r \xi(n_i\bm{\lambda}_i) \left(\prod_{j=i+1}^r \chi(\bm{\lambda}_i,\bm{\lambda}_j)^{n_in_j} \right)\\
    &=\prod_{i=1}^r \xi(\bm{\lambda}_i)^{n_i^2} \left(\prod_{j=i+1}^r \chi(\bm{\lambda}_i,\bm{\lambda}_j)^{n_in_j} \right).
\end{align*}
Combining the above two identities yields \eqref{W-formula}.
\medskip\par
\begin{proof}[Proof of Corollary \ref{cor-13}]
If $\kk=\FF_q$, a finite field with $q$ elements, and if $(q-1) | n_i$ for all $i$, then 
we get
$\xi(\sum_{i=1}^r n_i \bm{\lambda}_i) = 1$, since every term is raised to a power divisible by $n_i$
for some $i$. Similarly, $\chi(\bm{\lambda}_i,\vv)^{n_i}=1$.
\end{proof}
%Therefore, we get
%{\begin{cor}
%    Let $W:A \rightarrow \FF_q$ be an elliptic net,
%    and let $\Lambda$ have index at least $4$ in $A$. Assume that $\vv_1 \equiv \vv_2 \pmod{ (q-1)\Lambda}$. Then 
%    $W(\vv_1)=W(\vv_2)$.}
%\end{cor}
\begin{exam}\label{third-example}
    Here by an example we show that how one can use the identity (\ref{W-formula}) to calculate an arbitrary term of an elliptic net over a finite field. To illustrate the method we consider a rank $2$ elliptic net associated to an elliptic curve over $\mathbb{Q}$ and compute a specific term of its associated $p$-reduced nets as $p$ varies over certain primes.

    For a prime $p$ let $W:\ZZ^{2} \rightarrow \FF_{p}$ be the elliptic net associated to the rank $2$ elliptic curve $y^{2} = x^{3} - 11$ and generators $P = (3,4)$, and $Q = (15,58)$. The net $W$ has a unique rank of apparition respect to the standard basis $\{\ee_1, \ee_2\}$ and so its zero set forms a subgroup of rank $2$ of $\ZZ^2$. We choose a basis $\{\bm{\lambda}_1, \bm{\lambda}_2\}$ for this subgroup and by using definitions of functions $\xi$ and $\chi$ we compute 
    $\xi(\bm{\lambda}_{1})$, $\xi(\bm{\lambda}_{2})$,  $\chi(\bm{\lambda}_{1},\bm{\lambda}_{2})$, $\chi(\bm{\lambda}_{1}, \ee_{1})$, $\chi(\bm{\lambda}_{1}, \ee_{2})$, $\chi_(\bm{\lambda}_{2},\ee_{1})$, and $\chi(\bm{\lambda}_{2},\ee_{2})$. The following table summarizes the result of our computations for five values of $p$ (i.e. $p= 7, 11,  19, 61, 89$).
    \begin{table}[h!]
     \centering
     \begin{tabular}{c|c|c|c|c|c|c|c|c|c}
    $ p $ & $\bm{\lambda}_{1}$ & $\bm{\lambda}_{2}$ & $\xi(\bm{\lambda}_{1})$ & $\xi(\bm{\lambda}_{2})$ & $\chi(\bm{\lambda}_{1},\bm{\lambda}_{2})$ & $\chi(\bm{\lambda}_{1}, \ee_{1})$ & $\chi(\bm{\lambda}_{1}, \ee_{2})$ & $\chi_(\bm{\lambda}_{2},\ee_{1})$ & $\chi(\bm{\lambda}_{2},\ee_{2})$ \\
     \hline 
     7 & (1,5) & (0,13) & 1 & 4 & 3 & 3 & 3 & 6 & 2 \\
     11 & (1,7) & (0,11) & 4 & 9 & 9 & 4 & 9 & 9 & 6 \\
     19 & (1,6) & (0,14) & 8 & 5 & 4 & 1 & 3 & 6 & 2 \\
     61 & (2,8) & (0,38) & 39 & 60 & 19 & 34 & 6 & 43 & 41 \\
     89 & (9,3) & (0,10) & 87 & 43 & 80 & 62 & 58 & 52 & 33
     \end{tabular}
    \end{table}

    Let $D$ be a fixed set of representatives for $\ZZ^2/\Lambda$. Then any point $(r, s)$ in $\ZZ^2$ can be uniquely written as $(r, s)=n_{1}\bm{\lambda}_{1} + n_{2}\bm{\lambda}_{2} + m_{1}\ee_{1} + m_{2}\ee_{2}$ with $(m_1, m_2)\in D$.  Now by computing values for $W(m_1\ee_1+m_2\ee_2)$ (by using the defining recursion of our net), the above table,  and employing the rank $2$ version of \eqref{W-formula},
    \begin{multline*}
     W(n_{1}\bm{\lambda}_{1} + n_{2}\bm{\lambda}_{2} + m_{1}\ee_{1} + m_{2}\ee_{2})  =  \xi(\bm{\lambda}_{1})^{n_{1}^{2}}\xi(\bm{\lambda}_{2})^{n_{2}^{2}}\chi(\bm{\lambda}_{1},\bm{\lambda}_{2})^{n_{1}n_{2}} \chi(\bm{\lambda}_{1},\ee_{1})^{n_{1}m_{1}}\chi(\bm{\lambda}_{1},\ee_{2})^{n_{1}m_{2}}\\
    % \\ & & \times
    \times~\chi(\bm{\lambda}_{2},\ee_{1})^{n_{2}m_{1}}\chi(\bm{\lambda}_{2},\ee_{2})^{n_{2}m_{2}} W(m_{1}\ee_{1}+m_{2}\ee_{2}),
    \end{multline*}
    we can compute $W(r, s)$.

    Here by using the above formula and table we compute the term $W(101,100)$ modulo $p$.
    %\begin{table}[h!]
     %\centering
     %\begin{tabular}{c|c|c|c|c|c|c|c}
     %$p$ & LHS & $n_{1}$ & $n_{2}$ & $m_{1}$ & $m_{2}$ & $W(m_{1}\ee_{1}+m_{2}\ee_{2})$ & RHS \\
    %\hline
     %7 & 4 & 100 & -31 & 0 & 3 & 4 & 4 \\
     %11 & 7 & 99 & -31 & 1 & 2 & 4 & 7 \\
     %17 & 10 & 100 & -53 & 0 & 1 & 1 & 10 \\
     %19 & 12 & 99 & -66 & 1 & 1 & 1 & 12 \\
     %61 & 3 & 100 & -34 & 0 & 16 & 52 & 3
     %\end{tabular}
    %\end{table}

    \begin{table}[h!]
     \centering
     \begin{tabular}{c|c|c|c|c|c|c}
     $p$  & $n_{1}$ & $n_{2}$ & $m_{1}$ & $m_{2}$ & $W(m_{1}\ee_{1}+m_{2}\ee_{2})$ & W(101, 100) \\
    \hline
     7  & 101 & -32 & 0 & 11 & 3 & 1 \\
     11  & 101 & -56 & 0 & 9 & 6 & 5 \\
     19  & 101 & -37 & 0 & 12 & 12 & 12\\
     61  & 50 & -8 & 1 & 4 & 21 & 28 \\
     89  & 11 & 6 & 2 & 7 & 44 & 52  
     \end{tabular}
    \end{table}

\end{exam}
\section{Proofs of Proposition \ref{valuation-prop}, Theorem \ref{first-theorem}, and Proposition \ref{second-proposition}} 
\label{sec5}
Recall that $K$ is a field with a discrete valuation $\nu : K^\times \rightarrow \ZZ$.
We have $\OO$, $\pp$, and $\kk$  defined as before.
An application of the fact that $\Psi_\vv^\univ \in \RU_r^\univ$
is the following proof of proposition \ref{valuation-prop}.
\begin{proof}[Proof of Proposition \ref{valuation-prop}]
    %Let $E/K$ be an elliptic curve defined by the polynomial \eqref{WE} and let
    %$\mathbf{P} \in E(K)^r$ so that $P_i, P_i\pm P_j \neq \infty$.
    %Let $\pi_E : S^\univ \rightarrow K$ by sending $\alpha_i$ to $a_i$.
    %Recall that $\OO = \{ x \in K \st \nu(x) \geq 0\}$ and $\pp = \{x \in K \st \nu(x)\geq 1 \}.$
    %Assume now that $a_i \in \OO$. 
    Recall that $\pi_E: S^{\univ} \rightarrow K$ is defined by $\pi_E(\alpha_i)=a_i$. Then the image of $\pi_E$ lies in 
    $\OO$, so we can think of $\pi_E$ as a function from  $S^\univ$ into $\OO$.
    In particular for any $\vv \in \ZZ^r$ we get that
     \begin{equation}
     \label{net-polynomial}
     \Psi_\vv=(\pi_E)_*(\Psi_\vv^\univ) \in \OO[x_i,y_i]_{1\leq i\leq r}[(x_i-x_j)^{-1}]_{1\leq i <j\leq r}/{\langle f(x_i,y_i)\rangle}_{1\leq i \leq r}. 
     \end{equation}
    Now assume that $P_i \not \equiv \infty \pmod \pp$ and  
    $P_i \pm P_j \not \equiv \infty \pmod \pp$
    for all $i \neq j$.
    Then, since $P_i \not \equiv \infty \pmod \pp$, we have 
    $\nu(x(P_i)) \geq 0$ and $\nu(y(P_i)) \geq 0$ and so $\nu(x(P_i)-x(P_j)) \geq 0$.
    On the other hand, since $P_i, P_j, P_i \pm P_j \not \equiv \infty \pmod \pp$ we conclude that $x(P_i) \not \equiv x(P_j) \pmod \pp$, and thus 
    $\nu(x(P_i)-x(P_j)) \leq 0$. Therefore, $\nu(x(P_i)-x(P_j)) = 0$. This together with \eqref{net-polynomial} give
    $\nu(\Psi_\vv(\mathbf{P})) \geq 0$, as desired.
\end{proof}
\begin{proof}[Proof of Theorem \ref{first-theorem}] 
    $(a) \Longrightarrow (b)$. 
    Observe that $\Psi_{ne_i}({\bf P})=\psi_n(P_i)$.  So 
    the result follows from Theorem \ref{ayad}.  
    
    $(b) \Longrightarrow (c)$ is clear.  
    
    $(c) \Longleftrightarrow (d).$ 
    From Lemma \ref{numerator formula}, we have 
    \[ \Phi_\vv(\mathbf{P})=\Psi_\vv^2(\mathbf{P})x(P_i) - \Psi_{\vv+\ee_i}(\mathbf{P})\Psi_{\vv-\ee_i}(\mathbf{P}), \] 
    which implies the (c) and (d) are equivalent.

    $(c) \Longrightarrow (e)$. 
    First note that by proposition \ref{valuation-prop}, we have
    $\nu(\Psi_\vv(\mathbf{P})) \geq 0$, hence
    $\nu(\Psi_\vv(\mathbf{P})) \in \OO$ and therefore
    the reduction mod $\pp$ is well defined. We let
    $\Psi_\vv(\mathbf{P}) \pmod \pp$ be the image of $\Psi_\vv(\mathbf{P})$
    in the corresponding residue field under this reduction map.
    By part (a) of Lemma \ref{lem prelim} we get that
    $\Psi_\vv(\mathbf{P}) \pmod \pp$ is an elliptic net.
    Under the assumptions of (c) we have
    $\Psi_{\vv}(\mathbf{P}) \pmod \pp = 0$ and
    $\Psi_{\vv+\ee_i}(\mathbf{P})\pmod \pp =0$. Now if the zero set of 
    $\Psi_{\vv}(\mathbf{P}) \pmod \pp$ forms a subgroup then we have 
    $\Psi_{\ee_i}(\mathbf{P}) \pmod \pp=\psi_1(P_i) \pmod \pp=0$ 
    which is a contradiction, since $\psi_1=1$. So the zero set of 
    $\Psi_{\vv}(\mathbf{P}) \pmod \pp$ 
    does not form a subgroup of $\ZZ^r$ and thus by Theorem
    \ref{second-theorem} we conclude that 
    $\Psi_{\vv}(\mathbf{P}) \pmod \pp$ 
    does not have a unique rank of apparition (with respect to
    $\{\ee_1, \cdots, \ee_r\}$).
   % $\{(P_1, \infty, \dots, \infty), \dots, (\infty, \infty,\dots, P_r)\}$).  
    So there exists $1\leq i \leq r$ such that
    $\Psi_{n\ee_i}(\mathbf{P}) \pmod \pp$ does not have a unique
    rank of apparition.  By Theorem \ref{Ward6.2} we get
    that $\Psi_{3\ee_i} \pmod \pp =\Psi_{4\ee_i} \pmod \pp = 0,$
    which means
    $\nu(\Psi_{3\ee_i})$ and $\nu(\Psi_{4\ee_i}) > 0$.
    Therefore from Theorem \ref{ayad} we conclude that $P_i \pmod \pp$ is singular.

    $(e) \Longrightarrow (a)$ Since $P_i \pmod \pp$ is singular, then
    from Theorem \ref{ayad} we know that $\nu(\psi_2(P_i))>0$ and
    $\nu(\psi_3(P_i))>0$. Now the result follows since
    $\psi_n(P_i)=\Psi_{ne_i}({\bf P})$ for $n\in \ZZ$.  
\end{proof}

\begin{proof}[Proof of Proposition \ref{second-proposition}]
    First of all by \cite[Theorem 4.1]{AECII} if $P$ is a point such that $P \pmod p$ is non-singular then we have the following expression for the local N\'{e}ron height of $P$, 
$$\lambda_p(P)= \max\left\{ -{1\over 2} \nu_p(x(P)),0\right\}+\frac{1}{12}\nu_p(\Delta_E).$$  Observe that 
$$\nu_p(D_P)=\max\left\{-{1\over 2} \nu_p(x(P)), 0\right\}.$$ Under our assumptions since $P_i \pmod p$ is non-singular for $1\leq i \leq r$, we conclude that the quadratic form $\eps(\vv)$ in 
Lemma \ref{diff quad}, can be written as
    \[ \eps(\vv)=\nu_p(D_{\vv \cdot \mathbf{P}})-\nu_p(\Psi_\vv(\mathbf{P})) \]
for $\vv\neq \mathbf{0}$.
We also note that $\vv \mapsto \nu_p(F_\vv(\mathbf({P}))$ is a 
    quadratic form, where $F_\vv(\mathbf{P})$ is given in \eqref{FVP}.
    Define $\epshat:\ZZ^r \rightarrow \ZZ$ by
    \[ \epshat(\vv)=\eps(\vv)-\nu_p(F_\vv(\mathbf{P}))=
    \nu_p(D_{\vv \cdot \mathbf{P}})-\nu_p(\Psihat_\vv(\mathbf{P})). \]
    Since $\epshat$ is the difference of two quadratic forms, we conclude 
    that $\epshat$ is also a quadratic form.
    Furthermore, we have
    \begin{align*}
        \epshat(\ee_i)=\nu_p(D_{P_i})-\nu_p(\Psihat_{\ee_i}(\mathbf{P}))=0,
    \end{align*}
    for all $1\leq i \leq r$, 
    and
    \begin{align*}
        \epshat(\ee_i+\ee_j)=\nu_p(D_{P_i+P_j})-\nu_p(\Psihat_{\ee_i+\ee_j}(\mathbf{P}))=0,
    \end{align*}
    for all $1\leq i < j \leq r$.
    Thus by \cite[Lemma 4.5]{Stange} we have $\epshat(\vv)=0$ for
    all $\vv \in \ZZ^r$.
    This shows that, for all $\vv \in \ZZ^r$, we have
    \[ \nu_p(D_{\vv \cdot \mathbf{P}}) = \nu_p(\Psihat_\vv(\mathbf{P})), \]
    as desired.
\end{proof}

%\section{Examples}
The following two examples give illustrations of Proposition 
\ref{second-proposition}.
\begin{exam}{\label{first-example}}
    We consider the elliptic curve $E: y^2=x^3-11$. Then the group of rational
    points of $E$ over $\mathbb{Q}$ is generated by two points $P=(3, 4)$ and
    $Q=(15, 58)$.  We observe that $P, Q \not \equiv \infty \pmod p$ for all
    primes $p$ and $P+Q \not \equiv \infty \pmod p$ for all primes $p$ except
    $p=2$. In Table \ref{table1} we provide some values of the elliptic denominator net
    associated to $E$ and the points $P$ and $Q$ as  a two dimensional array with lower left corner  $D_{0Q+0P}$,  lower right corner 
    $D_{4Q+0P}$,  upper left corner  $D_{0Q+9P}$, and  upper right
    corner  $D_{4Q+9P}$. Table \ref{table2}  provides the corresponding values for the
    elliptic net associated to net polynomials $\Psi_{(v_1, v_2)}(P, Q)$. As
    predicted in Proposition \ref{second-proposition} the valuations of these
    two nets at all primes $p$ (except $p=2$) coincide.  
\end{exam}
\begin{exam}\label{second-example}
    We consider the elliptic curve $E: y^2+7y=x^3+x^2+28x$  with
    $E(\mathbb{Q})$ generated by two independent points $P=(0,0)$ and $Q=(1,
    3)$. Then $P, Q, P+Q \not \equiv \infty \pmod p$ for any prime $p$. However
    $P$ reduces to a singular point modulo $7$. Thus as predicted in Proposition
    \ref{second-proposition} the valuations of the elliptic denominator net (given in
    Table \ref{table3}) and the elliptic net (given in Table \ref{table4}) are the same for all
    primes $p\neq 7$.
\end{exam}

%\clearpage \clearpage 
%\newpage

\newpage
\newgeometry{bottom=0in}
\begin{landscape}
    
    \begin{table}
        \fontsize{6pt}{1}\selectfont
        \begin{tabular}{MMMMM}
            3^{3} \cdot 17 \cdot 861139 \cdot 638022143238323743 & 2 \cdot 31 \cdot 227 \cdot 32114101 \cdot 2233563433631 & 13 \cdot 97 \cdot 967 \cdot 2333 \cdot 899531 \cdot 20086489 & 2 \cdot 3^{2} \cdot 67 \cdot 89 \cdot 379 \cdot 1078019 \cdot 724929587 & 23 \cdot 103 \cdot 340789 \cdot 175849593114259\\
            2^{5} \cdot 37 \cdot 167 \cdot 245519 \cdot 3048674017 & 3 \cdot 7^{2} \cdot 11 \cdot 1567 \cdot 634026250609 & 2^{2} \cdot 5^{2} \cdot 43 \cdot 293 \cdot 349 \cdot 631 \cdot 1670527 & 41 \cdot 227 \cdot 4051 \cdot 32279374297 & 2^{3} \cdot 3 \cdot 17 \cdot 37 \cdot 47 \cdot 149 \cdot 263 \cdot 2003 \cdot 714947\\
            19 \cdot 433 \cdot 2689 \cdot 8819 \cdot 40487 & 2 \cdot 131 \cdot 179 \cdot 2103080101 & 3 \cdot 17 \cdot 101 \cdot 15641 \cdot 150379 & 2 \cdot 71 \cdot 83 \cdot 107 \cdot 751 \cdot 22613 & 77711 \cdot 82149276767\\
            2^{3} \cdot 3^{2} \cdot 5 \cdot 17 \cdot 23 \cdot 1737017 & 163 \cdot 1877 \cdot 42797 & 2^{2} \cdot 67 \cdot 317 \cdot 98377 & 3^{2} \cdot 5 \cdot 59 \cdot 25640299 & 2^{6} \cdot 7 \cdot 41 \cdot 157 \cdot 229 \cdot 9437\\
            449 \cdot 104759 & 2 \cdot 3 \cdot 29 \cdot 809 & 11 \cdot 19 \cdot 31 \cdot 677 & 2 \cdot 29 \cdot 569 \cdot 4987 & 3 \cdot 17 \cdot 1439 \cdot 925741\\
            2^{4} \cdot 37 \cdot 167 & 5^{2} \cdot 631 & 2^{2} \cdot 3 \cdot 17 \cdot 149 & 13 \cdot 30557 & 2^{3} \cdot 5 \cdot 37 \cdot 239 \cdot 1549\\
            3^{2} \cdot 17 & 2 \cdot 67 & 7 \cdot 157 & 2 \cdot 3^{3} \cdot 2087 & 19 \cdot 23 \cdot 503 \cdot 659\\
            2^{3} & 3 & 2^{2} \cdot 5 & 11 \cdot 1553 & 2^{4} \cdot 3 \cdot 17 \cdot 199 \cdot 577\\
            1 & 2 & 3 \cdot 17 & 2 \cdot 31 \cdot 233 & 631 \cdot 1753\\
            0 & 1 & 2^{2} \cdot 29 & 3^{2} \cdot 5 \cdot 3331 & 2^{3} \cdot 29 \cdot 37 \cdot 83 \cdot 3467\\ 
        \end{tabular}
      \bigskip\par  
        \caption{Elliptic denominator net associated to $E:y^2=x^3-11$ and the
        points $Q=(15,58)$ and $P=(3, 4)$}
        \label{table1}
    \end{table}

    \begin{table} 
        \fontsize{6pt}{1}\selectfont
        \begin{tabular}{MMMMM}
                -3^{3} \cdot 17 \cdot 861139 \cdot 638022143238323743 & -2^{-8} \cdot 31 \cdot 227 \cdot 32114101 \cdot 2233563433631 & -2^{-18} \cdot 13 \cdot 97 \cdot 967 \cdot 2333 \cdot 899531 \cdot 20086489 & -2^{-26} \cdot 3^{2} \cdot 67 \cdot 89 \cdot 379 \cdot 1078019 \cdot 724929587 & -2^{-36} \cdot 23 \cdot 103 \cdot 340789 \cdot 175849593114259 \\
                2^{5} \cdot 37 \cdot 167 \cdot 245519 \cdot 3048674017 & -2^{-8} \cdot 3 \cdot 7^{2} \cdot 11 \cdot 1567 \cdot 634026250609 & -2^{-14} \cdot 5^{2} \cdot 43 \cdot 293 \cdot 349 \cdot 631 \cdot 1670527 & -2^{-24} \cdot 41 \cdot 227 \cdot 4051 \cdot 32279374297 & -2^{-29} \cdot 3 \cdot 17 \cdot 37 \cdot 47 \cdot 149 \cdot 263 \cdot 2003 \cdot 714947\\
                19 \cdot 433 \cdot 2689 \cdot 8819 \cdot 40487 & 2^{-6} \cdot 131 \cdot 179 \cdot 2103080101 & 2^{-14} \cdot 3 \cdot 17 \cdot 101 \cdot 15641 \cdot 150379 & -2^{-20} \cdot 71 \cdot 83 \cdot 107 \cdot 751 \cdot 22613 & -2^{-28} \cdot 77711 \cdot 82149276767\\
                2^{3} \cdot 3^{2} \cdot 5 \cdot 17 \cdot 23 \cdot 1737017 & 2^{-6} \cdot 163 \cdot 1877 \cdot 42797 & 2^{-10} \cdot 67 \cdot 317 \cdot 98377 & 2^{-18} \cdot 3^{2} \cdot 5 \cdot 59 \cdot 25640299 & 2^{-18} \cdot 7 \cdot 41 \cdot 157 \cdot 229 \cdot 9437\\
                -449 \cdot 104759 & -2^{-4} \cdot 3 \cdot 29 \cdot 809 & 2^{-10} \cdot 11 \cdot 19 \cdot 31 \cdot 677 & 2^{-14} \cdot 29 \cdot 569 \cdot 4987 & 2^{-20} \cdot 3 \cdot 17 \cdot 1439 \cdot 925741\\
                -2^{4} \cdot 37 \cdot 167 & -2^{-4} \cdot 5^{2} \cdot 631 & -2^{-6} \cdot 3 \cdot 17 \cdot 149 & -2^{-12} \cdot 13 \cdot 30557 & 2^{-13} \cdot 5 \cdot 37 \cdot 239 \cdot 1549\\
                -3^{2} \cdot 17 & -2^{-2} \cdot 67 & -2^{-6} \cdot 7 \cdot 157 & -2^{-8} \cdot 3^{3} \cdot 2087 & -2^{-12} \cdot 19 \cdot 23 \cdot 503 \cdot 659\\
                2^{3} & 2^{-2} \cdot 3 & -2^{-2} \cdot 5 & -2^{-6} \cdot 11 \cdot 1553 & -2^{-4} \cdot 3 \cdot 17 \cdot 199 \cdot 577\\
                1 & 1 & 2^{-2} \cdot 3 \cdot 17 & 2^{-2} \cdot 31 \cdot 233 & -2^{-4} \cdot 631 \cdot 1753\\
                0 & 1 & 2^{2} \cdot 29 & 3^{2} \cdot 5 \cdot 3331 & 2^{3} \cdot 29 \cdot 37 \cdot 83 \cdot 3467\\ 
        \end{tabular}
        \bigskip\par
        \caption{Elliptic net associated to $E: y^2 = x^3 -11$ and the points $Q = (15,58)$ and $P=(3, 4)$.}
        \label{table2}
    \end{table}

\clearpage
%\newgeometry{top=0in}
%\thispagestyle{empty}
    \begin{table}
        \centering
        \fontsize{6pt}{1}\selectfont
        \begin{tabular}{MMMMMMM}
              3^{2} \cdot 5 \cdot 8243 \cdot 7289363 & 59 \cdot 523 \cdot 1170779 & 2803 \cdot 2163467 & 2^{3} \cdot 23 \cdot 7758139 & 59 \cdot 149837011 & 31 \cdot 229 \cdot 32045369 & 3 \cdot 11 \cdot 733 \cdot 154099559\\
              13 \cdot 127 \cdot 3066533 & 2 \cdot 41 \cdot 53 \cdot 26627 & 7 \cdot 13 \cdot 17 \cdot 5653 & 5^{2} \cdot 29 \cdot 67 \cdot 487 & 3 \cdot 13 \cdot 19 \cdot 89 \cdot 1291 & 7 \cdot 109 \cdot 1427 \cdot 2833 & 2^{2} \cdot 13 \cdot 167 \cdot 199 \cdot 617887\\
              5948431 & 181 \cdot 8819 & 3^{2} \cdot 47 \cdot 1097 & 11 \cdot 11779 & 2 \cdot 61 \cdot 74377 & 17 \cdot 25967671 & 5 \cdot 56479 \cdot 333271\\
              3 \cdot 5 \cdot 7 \cdot 1949 & 6553 & 2^{4} \cdot 431 & 7^{2} \cdot 521 & 42181 & 47 \cdot 71 \cdot 14557 & 3 \cdot 7 \cdot 127 \cdot 349 \cdot 32537\\
              2 \cdot 11 \cdot 113 & 911 & 463 & 5 \cdot 557 & 3^{3} \cdot 37 \cdot 137 & 2^{2} \cdot 2059769 & 25084117199\\
              127 & 7 & 3 \cdot 19 & 2 \cdot 199 & 7 \cdot 2039 & 653 \cdot 15767 & 5 \cdot 11 \cdot 293 \cdot 662327\\
              3 \cdot 5 & 2^{3} & 1 & 349 & 53 \cdot 593 & 5624039 & 2 \cdot 3^{2} \cdot 41 \cdot 73 \cdot 661 \cdot 2141\\
              1 & 1 & 7 & 5 \cdot 11 & 2^{2} \cdot 3 \cdot 23 \cdot 107 & 7 \cdot 4812433 & 19 \cdot 127 \cdot 601 \cdot 4637\\
              1 & 1 & 2 \cdot 3 & 601 & 277 \cdot 313 & 1987 \cdot 119321 & 5^{2} \cdot 139843540153\\
              0 & 1 & 13 & 7 \cdot 59 & 13 \cdot 55819 & 2 \cdot 29 \cdot 26272439 & 3 \cdot 7 \cdot 13 \cdot 59 \cdot 263 \cdot 5880307\\ 
        \end{tabular}
        \bigskip\par
        \caption{Elliptic denominator net associated to $E: y^2 +7y = x^3 +x^2 +28x$ and the points $Q = (1,3)$ and $P=(0, 0)$}
        \label{table3}
    \end{table}

    \begin{table}
        \centering
        \fontsize{6pt}{1}\selectfont
        \begin{tabular}{MMMMMMM} 
            -3^{2} \cdot 5 \cdot 7^{20} \cdot 8243 \cdot 7289363 & 7^{20} \cdot 59 \cdot 523 \cdot 1170779 & 7^{20} \cdot 2803 \cdot 2163467 & 2^{3} \cdot 7^{20} \cdot 23 \cdot 7758139 & -7^{20} \cdot 59 \cdot 149837011 & -7^{20} \cdot 31 \cdot 229 \cdot 32045369 & -3 \cdot 7^{20} \cdot 11 \cdot 733 \cdot 154099559\\
            -7^{16} \cdot 13 \cdot 127 \cdot 3066533 & -2 \cdot 7^{16} \cdot 41 \cdot 53 \cdot 26627 & 7^{17} \cdot 13 \cdot 17 \cdot 5653 & 5^{2} \cdot 7^{16} \cdot 29 \cdot 67 \cdot 487 & 3 \cdot 7^{16} \cdot 13 \cdot 19 \cdot 89 \cdot 1291 & -7^{17} \cdot 109 \cdot 1427 \cdot 2833 & -2^{2} \cdot 7^{16} \cdot 13 \cdot 167 \cdot 199 \cdot 617887\\
            7^{12} \cdot 5948431 & -7^{12} \cdot 181 \cdot 8819 & -3^{2} \cdot 7^{12} \cdot 47 \cdot 1097 & 7^{12} \cdot 11 \cdot 11779 & 2 \cdot 7^{12} \cdot 61 \cdot 74377 & 7^{12} \cdot 17 \cdot 25967671 & -5 \cdot 7^{12} \cdot 56479 \cdot 333271\\
            3 \cdot 5 \cdot 7^{10} \cdot 1949 & 7^{9} \cdot 6553 & -2^{4} \cdot 7^{9} \cdot 431 & -7^{11} \cdot 521 & -7^{9} \cdot 42181 & 7^{9} \cdot 47 \cdot 71 \cdot 14557 & 3 \cdot 7^{10} \cdot 127 \cdot 349 \cdot 32537\\
            2 \cdot 7^{6} \cdot 11 \cdot 113 & 7^{6} \cdot 911 & 7^{6} \cdot 463 & -5 \cdot 7^{6} \cdot 557 & -3^{3} \cdot 7^{6} \cdot 37 \cdot 137 & -2^{2} \cdot 7^{6} \cdot 2059769 & 7^{6} \cdot 25084117199\\
            -7^{4} \cdot 127 & 7^{5} & 3 \cdot 7^{4} \cdot 19 & 2 \cdot 7^{4} \cdot 199 & -7^{5} \cdot 2039 & -7^{4} \cdot 653 \cdot 15767 & -5 \cdot 7^{4} \cdot 11 \cdot 293 \cdot 662327\\
            -3 \cdot 5 \cdot 7^{2} & -2^{3} \cdot 7^{2} & -7^{2} & 7^{2} \cdot 349 & 7^{2} \cdot 53 \cdot 593 & -7^{2} \cdot 5624039 & -2 \cdot 3^{2} \cdot 7^{2} \cdot 41 \cdot 73 \cdot 661 \cdot 2141\\
            7 & -7 & -7^{2} & -5 \cdot 7 \cdot 11 & 2^{2} \cdot 3 \cdot 7 \cdot 23 \cdot 107 & 7^{2} \cdot 4812433 & -7 \cdot 19 \cdot 127 \cdot 601 \cdot 4637\\
            1 & 1 & -2 \cdot 3 & -601 & -277 \cdot 313 & 1987 \cdot 119321 & 5^{2} \cdot 139843540153\\
            0 & 1 & 13 & -7 \cdot 59 & -13 \cdot 55819 & -2 \cdot 29 \cdot 26272439 & 3 \cdot 7 \cdot 13 \cdot 59 \cdot 263 \cdot 5880307\\ 
        \end{tabular}
        \bigskip\par
        \caption{Elliptic nets associated to $E: y^{2} +7y = x^{3} +x^{2} +28x$ and the points  $Q = (1,3)$ and $P=(0, 0)$}
        \label{table4}
    \end{table}

\end{landscape}

\newpage

\end{document}